\documentclass[journal]{IEEEtran}

\usepackage[ruled,linesnumbered]{algorithm2e}
\usepackage{subfigure}
\usepackage{epsfig} 
\usepackage{mathptmx} 
\usepackage{times} 
\usepackage{amsmath} 
\usepackage{amssymb}  
\usepackage{amsthm}
\usepackage{makeidx}
\usepackage{float}
\usepackage{balance}
\usepackage{booktabs}
\usepackage{bbm}
\usepackage{bm}
\usepackage{caption}
\usepackage{nomencl}
\usepackage{dblfloatfix}

\newtheorem{theorem}{Theorem}[section]
\newtheorem{lemma}{Lemma}[section]
\newtheorem{remark}{Remark}[section]
\newtheorem{prop}{Proposition}[section]
\newtheorem{ass}{Assumption}[section]
\IEEEoverridecommandlockouts                              
\usepackage{etoolbox}
\makeatletter
\patchcmd{\@makecaption}
  {\scshape}
  {}
  {}
  {}
\makeatother

\newcommand{\rd}{{\mathrm d}}

\usepackage[]{algpseudocode}

\begin{document}
%
\title{Stochastic Trajectory Optimization for Mechanical Systems with Parametric Uncertainties}
%
%
%
%

\author{ George I. Boutselis$^{1}$, Yunpeng Pan$^{1}$, Gerardo De La Torre$^{2}$ and Evangelos A.   Theodorou$^{3}$
\thanks{$^{1}$ George I. Boutselis and Yunpeng Pan are PhD graduate students in the School of Aerospace Engineering at Georgia Institute of Technology
     }
     \thanks{$^{2}$ Gerardo De La Torre is postdoctoral researcher in Northwestern University
     }   
  \thanks{$^{3}$ Evangelos A. Theodorou is  Assistant Professor with the School of Aerospace Engineering at Georgia Institute of Technology
     }   
}

\markboth{Journal of \LaTeX\ Class Files,~Vol.~13, No.~9, September~2014}%
{Shell \MakeLowercase{\textit{et al.}}: Bare Advanced Demo of IEEEtran.cls for Journals}

\IEEEtitleabstractindextext{%
\begin{abstract}

 In this paper we develop a novel, discrete-time optimal control framework for mechanical systems with uncertain model parameters. We consider finite-horizon problems where the performance index depends on the statistical moments of the stochastic system. Our approach constitutes an extension of the original Differential Dynamic Programming method and handles uncertainty through generalized Polynomial Chaos (gPC) theory. The developed iterative scheme is capable of controlling the probabilistic evolution of the dynamic system and can be used in planning and control. Moreover, its scalable and fast-converging nature play a key role when dealing with complex, high-dimensional problems.
 
 Based on Lagrangian mechanics principles, we also prove that Variational Integrators can be designed to properly propagate and linearize the gPC representations of stochastic, forced mechanical systems in discrete time. We utilize this benefit to further improve the efficiency of our trajectory-optimization methodology. Numerical simulations are included to validate the applicability of our approach.

\end{abstract}

}

\maketitle


\IEEEdisplaynontitleabstractindextext

%
\IEEEpeerreviewmaketitle

\section*{Nomenclature}
\addcontentsline{toc}{section}{Nomenclature}
\begin{IEEEdescription}[\IEEEusemathlabelsep]
\item[$\xi$] Set of mutually independent random variables
\item[$\rho$] Probability density function of $\xi$
\item[$\phi$] Polynomials orthogonal with respect to $\rho$
\item[$x$] State vector
\item[$u$] Control vector
\item[$f$] Equations of dynamics
\item[$L$] Lagrangian
\item[$F$] Set of non-conservative forces
\item[$q$] Generalized position coordinates
\item[$p$] Generalized momentum coordinates
\item[$\textbf X$] Polynomial Chaos expansion coefficients of the state vector
\item[$\textbf f$] Polynomial Chaos representation of dynamics
\item[$\hat L$] Polynomial Chaos representation of the Lagrangian
\item[$\hat {\textbf F}$] Polynomial Chaos representation of non-conservative forces
\item[$\textbf Q$] Polynomial Chaos expansion coefficients of $q$
\item[$\textbf P$] Polynomial Chaos expansion coefficients of $p$
\item[$V$] Value function for the optimal control problem
\item[$Q$] $Q$-function used in gPC-DDP
\item[$\bm L$] Running cost for the Polynomial Chaos-based optimal control problem
\item[$\bm F$] Terminal cost for the Polynomial Chaos-based optimal control problem
\item[$\bm J$] Total cost for the Polynomial Chaos-based optimal control problem
\item[$\mathsf f^k$] Discrete version of the argument $\mathsf f$, evaluated at the $k^\text{th}$ time instant
\item[$D_i\mathsf f(\cdot)$] Partial derivative of a function $\mathsf f$ with respect to its $i^\text{th}$ argument 
\end{IEEEdescription}

\section{Introduction}
One of the major challenges in robotics is having autonomous systems behave efficiently under the presence of uncertainty. This can be mathematically formulated as optimizing certain performance criteria, which are relevant to the task under consideration and the system itself. Such problems have been addressed in the optimal control discipline. Broadly speaking, optimal control frameworks can be classified as model-free, or model-based methods. As its name implies, a model-free approach can be applied without having any information about the system dynamics. This fact has allowed reinforcement learning methods to be successfully applied in numerous robotic tasks (e.g., \cite{rl,pisq,Freek2012_SequenceManipulation,Buchli_IJRR_2011,PetersSchaal2008,KoberPeters2008,PetersThesis}) and to bypass the issue of model uncertainty. However, they suffer from rather slow convergence and require executing many rollouts on the physical system. On the other hand, model-based methods rely on a mathematical representation of the considered system (see for example \cite{ddp,ilqr,Lantoine2012,Lin1991_Part1,Lin1991_Part2}), resulting in much faster algorithms. Nevertheless, their effectiveness is largely dependent on modeling accuracy. To this end, a number of stochastic optimal control algorithms has been proposed in the literature that handle this shortcoming through various ways of uncertainty representation.

One class of model-based optimal control methods represents uncertainty by employing stochastic differential equations. 
Some of the most popular approaches that can deal with non-linear dynamics, include iterative linearization methods such as iterative Linear Quadratic Gaussian (iLQG) control \cite{ilqg} and Stochastic Differential Dynamic Programming (SDDP) \cite{sddp}, as well as sampling methods like the Path Integral control (PI) \cite{pathinteg}. Unfortunately, the aforementioned works rely on certain strong assumptions that may reduce their applicability. The SDDP and iLQG methodologies assume that the underlying dynamics have the form of an ordinary differential equation driven by white Gaussian noise with additive or multiplicative amplitude. Note that this way of modeling cannot capture stochasticity directly in the internal parameters, from which many times uncertainty stems from. In addition, continuous-time white noise signals possess infinite energy and, therefore, do not exist in the physical world. The PI methodology further assumes that the system dynamics are affine in control and requires that control authority is proportional to noise intensity.

Recently, novel probabilistic trajectory optimization algorithms \cite{pilco}, \cite{lolpan2014probabilistic} were developed for control of uncertain systems using Gaussian Processes (GP) \cite{gpr}.  Gaussian Processes provide a way to learn and represent unknown functions  by storing all the data collected from the system in consideration and performing inference. Due to the probabilistic nature of GP-regression,   GP-based trajectory optimization methods can be used for learning control on systems with completely unknown dynamics  while  reducing the number of interactions with the physical system during training. This important feature makes the aforementioned methods suitable for applications in which no prior knowledge of the dynamics is available and experiments are expensive. Unfortunately, most  GP-based trajectory optimization methods achieve tractability in the inference phase by treating the state as a Gaussian random variable and relying on moment matching (e.g., in \cite{pilco}). Since the Gaussianity assumption is not usually satisfied for non-linear systems, this may lead to erroneous representations. In addition, it is often the case in robotics that uncertainty arises due to unknown parameters in the dynamics. Thus, exploiting the structure of a physics-based model can provide great efficiency.

One method that has been widely used for representing parametric uncertainty in engineering systems, is Polynomial Chaos theory. Wiener introduced Polynomial Chaos \cite{wiener} and used it to decompose stochastic processes into a convergent series of Hermite polynomials. Xiu and Karniadakis in \cite{xiu} extended this concept and introduced the generalized Polynomial Chaos (gPC) framework. In their work, various continuous and discrete distributions were modeled using orthogonal polynomials from the Askey-scheme and $L^2$ convergence in the corresponding Hilbert functional space was guaranteed. Based on their pioneering work, gPC has been successfully applied in various fields such as solid mechanics \cite{solidm} and fluid dynamics \cite{fluidd}. However, works including applications on autonomous systems and control-related problems are rather limited. The authors in \cite{vdyn} used the gPC scheme to analyze the dynamics of a vehicle under uncertainty. Dutta and Bhattacharya in \cite{dutta} designed nonlinear estimators which proved to be beneficial when measurements are infrequent. Hover and Triantafyllou utilized gPC as a tool to analyze the stability of a stochastic bilinear system \cite{hover}. Moreover, Fisher and Bhattacharya in \cite{fisher} proposed a stochastic version of the LQR controller.

In our work, we consider dynamic systems which are influenced by a set of uncertain internal parameters and initial states. Specifically, we assume that the exact values of these quantities are not known, but we have access to their underlying statistics. We utilize the gPC framework to model this type of uncertainty and transform the, originally, stochastic system into a set of deterministic ordinary differential equations that capture uncertainty evolution over time. This allows us to formulate classes of deterministic optimal control problems which depend directly on the gPC representation of the dynamics and, therefore, have a stochastic flavor. To solve the obtained problems, we propose a novel version of Differential Dynamic Programming (DDP) (\cite{ddp}, \cite{conlddp}) and call the developed algorithm ``gPC-DDP". Our framework is an iterative, scalable method which is able to control the probabilistic behavior of the trajectory in an optimization setting. In addition, we prove that under some mild assumptions, gPC-DDP admits locally quadratic convergence rates for generic problem formulations.

The performance of the proposed framework is further improved by incorporating the concept of Variational Integrators (VI's). VI's are a class of numerical time stepping methods derived from direct discretization of Hamilton's principle \cite{marsden}, \cite{varanim}. They have been shown to outperform standard numerical integration schemes (e.g. Euler differentiation, Runge-Kutta methods) due to their long-term energy preserving properties. Specifically, one major benefit is their insensitivity to both time step and terminal time selection \cite{marsden}, \cite{ima}. This can be extremely useful for robotics applications when a long time horizon is considered and fast online computations have to be made. In addition, their available structured linearization renders them ideal candidates to use when numerically solving optimal control problems \cite{lin}. Inspired by the work in \cite{pasini}, we prove that Polynomial Chaos representations of mechanical systems will obey the laws of Lagrangian mechanics. By providing explicit expressions for the Lagrangian and non-conservative forces (e.g., control inputs, dissipative forces) we are able to develop variational integration schemes and incorporate them in our discrete-time optimal control methodology.

The contribution of this work lies in developing a generic framework for controlling the probabilistic evolution of systems under parametric uncertainties. In contrast to most existing approaches, limiting assumptions on the dynamics structure and type of stochasticity can be avoided in the problem setup. Moreover, our methodology utilizes the benefits of the original Differential Dynamic Programming algorithm towards obtaining tractable solutions for nonlinear, stochastic optimal control problems. We also provide a convergence proof that generalizes prior theoretical work, by considering problems with generic running and terminal cost function terms. Our analysis, therefore, completes past work on the convergence properties of DDP-based algorithms. Last but not least, when Lagrangian systems are considered, we show that Variational Integrators can be employed to achieve superior numerical performance during uncertainty propagation. The aforementioned properties highlight the potential applicability of the proposed approach in robotics problems.

The remaining of this paper is organized as follows. Section \ref{secprel} briefly states the problem we will be dealing with and also includes some preliminaries. Regarding the latter, subsection \ref{secgpc} gives an overview of the Wiener-Askey Polynomial Chaos scheme, while subsection \ref{secvi} provides some basic elements of discrete mechanics and variational integrators. In section \ref{secvigpc}, we discuss Polynomial Chaos representations of dynamic systems, for which we also develop a Variational Integrator. Section \ref{secgpcddp} derives the gPC-DDP framework and highlights its important features. Simulated examples are included in section \ref{secsim} that provide further insight into the behavior of our algorithm. Section \ref{secdisc} discusses extensions that can deal with possible drawbacks of our methodology and section \ref{secconc} is the conclusion. Lastly, the appendices include some technical details as well as the core of gPC-DDP's convergence proof.

\section{Problem formulation and preliminaries}\label{secprel}

\subsection{Problem statement}
We consider dynamic systems of the form
\begin{equation}
\label{dyn1}
\dot{x}(t)=f(x(t),u(t),t;\lambda^p),
\end{equation}
where $x\in\mathbb{R}^n$ is the state vector, $u\in\mathbb{R}^m$ is the control input vector and $\lambda^p$ denotes a set of independent random parameters for the physics-based model \eqref{dyn1}. We also consider the initial state $x^0$ to be a function of a set of independent random variables $\lambda^0$. It is assumed that the statistics of $\lambda^p$ and $\lambda^0$ are known a priori. The goal is to find a control $u(\cdot)$ that solves the following stochastic optimal control problem

\begin{equation}
\label{smxi}
\begin{split}
&\min_u\hspace{1.3mm} \int_{t_0}^{t_f}\mathcal{L}(\mathcal{M}(t),u,t)\rd t+\mathcal{F}(\mathcal{M}(t_f),t_f)\\
\text{s.t.}&\quad\dot{x}(t)=f(x(t),u(t),t;\lambda^p),\quad x^0=x(t_0;\lambda^0),
\end{split}
\end{equation}
where $[t_0,t_f]$ is the time horizon, $\mathcal{L}$ is the running cost and $\mathcal{F}$ is the terminal cost. In addition, $\mathcal{M}=(\mathcal{M}_1,...,\mathcal{M}_n)^\top$ denotes the moments of the state vector $x$, with $\mathcal{M}_i=((\mathcal{M}_i)_1,...,(\mathcal{M}_i)_j,...)$ and $(\mathcal{M}_i)_j$ being the $j^\text{th}$ (central) moment of state $x_i$. Henceforth, we assume that a solution to \eqref{smxi} exists.

In this paper, we will deal with the discrete-time counterpart of \eqref{smxi} by discretizing the corresponding cost functions and dynamic equations. Moreover, in sections \ref{secvigpc}, \ref{secgpcddp}, we will utilize Polynomial Chaos theory to transform the stochastic optimal control problem in \eqref{smxi} (usually referred to as moment-based stochastic control problem; see for example \cite{xu1}, \cite{xu3}) into a purely deterministic one.

Finally, note that solving optimal control problems with continuous-valued states and controls can rarely be done analytically. In addition, when high-dimensional, non-linear systems are considered, searching for global optimality becomes usually intractable. Hence, we seek to develop numerical methods that can produce locally optimal, admissible and tractable solutions, at the expense of sacrificing global optimality.

\subsection{Wiener-Askey Polynomial Chaos}\label{secgpc}
In this subsection we review the basics of generalized Polynomial Chaos theory. More details can be found in \cite{bookxiu}.

Let $(\Omega,\mathcal{F},P)$ be a probability space such that $\Omega$ is the sample space, $\mathcal{F}$ is the $\sigma$-field of $\Omega$ and $P$ is the corresponding probability measure. Let $\omega\in\Omega$ and suppose that $\xi(\omega)=(\xi_1(\omega),...,\xi_d(\omega))\in\mathbb{R}^d$ is a continuous multi-dimensional random variable with mutually independent components. Define $\rho(\xi)$ to be a positive weight function and determine the weighted $L^2_\rho$ space by
\begin{equation}
\label{Lro}
L^2_\rho=\bigg\{\mathsf f:\mathcal{D}\rightarrow\mathbb{R}\bigg\vert\int_\mathcal{D} \mathsf f^2(\xi)\rho(\xi)\rd\xi<\infty\bigg\},
\end{equation}
where $\mathcal{D}$ is the support of random variable $\xi$. In the context of Polynomial Chaos theory, $\rho(\xi)$ is related to the probability density function associated with $P$ (i.e., $\rd P(\xi)=Z\rho(\xi)\rd\xi$, where $Z$ is a scaling factor). Thus, $L^2_\rho$ denotes the space of mean-square integrable functions (i.e., functions with finite second moment). One can then write the Polynomial Chaos expansion of $\mathsf f\in L^2_\rho$ as
\begin{equation}
\label{gpcexpansion}
\mathsf f(\xi)=\sum_{j=0}^\infty \mathsf f_j\phi_j(\xi).
\end{equation}
The set $\{\phi_j(\xi),\hspace{1mm}j\in\mathbb{Z}_{\geq0}\}$ consists of polynomials orthogonal with respect to the weight function $\rho(\xi)$, while $\{\mathsf f_j,\hspace{1mm}j\in\mathbb{Z}_{\geq0}\}$ contains the coefficients of the Polynomial Chaos expansion. It has been shown that the series in \eqref{gpcexpansion} converges to $\mathsf f$ in the $L^2_\rho$ sense \cite{bookxiu}. However, while the requirement for convergence is rather mild ($L_\rho^2$-integrability), the rate of convergence will further depend on the regularity of $\mathsf f$ with respect to $\xi$.

In practice, we will have to truncate the series as
\begin{equation}
\label{gpcexpansion1}
\mathsf f(\xi)\approx\sum_{j=0}^K\mathsf f_j\phi_j(\xi),
\end{equation}
where the number of coefficients is $K=\frac{(r+d)!}{r!d!}-1$, such that $r$ is the maximum order of $\{\phi_j\}$ and $d$ is the dimension of $\xi$. Since independence among random inputs is assumed, when $d>1$, the elements of $\{\phi_j\}$ are defined as products of univariate orthogonal polynomials. Specifically, let $Z\rho(\xi)=Z_1\rho_1(\xi_1)...Z_d\rho_d(\xi_d)$ be the joint probability density function. Then, $\phi_j(\xi)=\tilde{\phi}_{j_1}(\xi_1)...\tilde{\phi}_{j_d}(\xi_d)$, where $\tilde{\phi}_{j_i}$ denotes polynomials of order $j_i$ that are orthogonal with respect to $\rho_i(\xi_i)$ ($j_i\leq r$, $j=0,...,K$ and $i=1,...,d$).

The expansion in \eqref{gpcexpansion1} can be viewed as an orthogonal projection of $\mathsf f$ onto the linear space spanned by $\{\phi_j,\hspace{1mm}j=0,...,K\}$. It has been proven that for any $\mathsf f\in L^2_\rho$, this orthogonal projection constitutes the best polynomial approximation in the $L^2_\rho$ norm \cite{bookxiu}.

To proceed, the $j^\textrm{th}$ coefficient in \eqref{gpcexpansion1} can be obtained by using the orthogonality property of $\{\phi_j\}$. Specifically one has
\begin{equation}
\label{galerkin}
\mathsf f_j=\frac{\int_\mathcal{D}\mathsf f(\xi)\phi_j(\xi)\rho(\xi)\rd\xi}{\int_\mathcal{D}\phi_j^2(\xi)\rho(\xi)\rd\xi}.
\end{equation}
This procedure is called Galerkin projection and virtually requires the residual error of the projection to be orthogonal to the space spanned by $\{\phi_j,\hspace{1mm}j=0,...,K\}$.

The authors in \cite{xiu} established a connection between distributions of random inputs and orthogonal polynomials from the Askey scheme, developing the generalized Polynomial Chaos framework. They showed that when $\rho(\xi)$ is of a certain type, one can naturally select the appropriate set $\{\phi_j\}$ from the Askey scheme that gives orthogonality with respect to $\rho(\xi)$. Table \ref{tab:ask} shows this correspondence and provides a list of such orthogonal polynomials\footnote{$\alpha$, $\beta$ in Table \ref{tab:ask} denote parameters of the corresponding density functions}. Each set $\{\phi_j,\hspace{1mm}j\in\mathbb{Z}_{\geq0}\}$ forms a complete orthogonal basis in the Hilbert space determined by the corresponding weight function. Moreover, when the density function of the quantity to be approximated belongs in Table \ref{tab:ask}, proper selection of the basis functions results in faster convergence to the desired distribution.
\begin{table}[h]
  \centering
  \caption{Correspondence between standard forms of continuous probability distributions and types of continuous orthogonal polynomials from the Askey scheme}
  \label{tab:ask}
  \begin{tabular}{cccc}
    \toprule
    \normalsize Distribution & \normalsize Weight function & \normalsize Polynomials & \normalsize Domain\\
\midrule
Gaussian & $e^{-\xi^2/2}$ &  Hermite &  $(-\infty,\infty)$\\

 Uniform & $1$ &  Legendre &  $[-1,1]$\\

Gamma & $\xi^\alpha e^{-\xi}$ &  Laguerre &  $[0,\infty)$\\

 Beta & $(1-\xi)^\alpha(1+\xi)^\beta$ &  Jacobi &  $[-1,1]$\\
    \bottomrule
  \end{tabular}
\end{table}

In the remaining of this paper, we will often denote the expectation operator with respect to $\rho(\xi)$ as
\begin{equation*}
\begin{split}
\langle\mathsf f\rangle&=\mathbb{E}[\mathsf f]=\int_\mathcal{D}\mathsf f(\xi)\rho(\xi)\rd \xi,\\
\langle\mathsf f,\mathsf g\rangle&=\mathbb{E}[\mathsf f\mathsf g]=\int_\mathcal{D}\mathsf f(\xi)\mathsf g(\xi)\rho(\xi)\rd \xi
\end{split}
\end{equation*}
and so forth. $\langle\mathsf f,\mathsf g\rangle$ can be viewed as the inner product in the Hilbert space of mean-square integrable functions. Also, due to orthogonality we have
\begin{equation*}
\langle\phi_i,\phi_j\rangle=\delta_{ij}\langle\phi_i,\phi_i\rangle,
\end{equation*}
where $\delta_{ij}$ is the Kronecker delta.

Lastly note that Table \ref{tab:ask} provides details for continuous random variables. Analogously, if discrete random variables are considered, certain orthogonal polynomials from the Askey scheme can be used \cite{bookxiu}. In addition, arbitrary probability distributions can be handled as well \cite{gpca}.

\subsection{Discrete mechanics and Variational Integrators}\label{secvi}
Here, we present fundamental concepts of Lagrangian mechanics in the continuous and discrete time domain. This involves discussing the Pontryagin-d'Alembert principle, as well as the Discrete Euler-Lagrange equations. These elements will be used in subsequent sections to develop a Variational Integrator for Polynomial Chaos representations of dynamical systems. We will consider the case where non-conservative forces are applied since it is more relevant to robotics applications. The reader is referred to \cite{marsden}, \cite{varanim} for more details.

\textbf{The continuous Lagrange-d'Alembert principle} Consider a finite-dimensional dynamical system and let $q=(q_1,...,q_N)^\top\in\mathbb{R}^N$ denote its generalized position coordinates and $L(q,\dot{q})$ its Lagrangian. 
Under the influence of non-conservative forces $F=(F_1,...,F_N)^\top\in\mathbb{R}^N$, the Lagrange-d'Alembert principle states that over a time horizon $t\in[t_0,t_f]$ the following holds
\begin{equation}
\label{dale}
\delta\int_{t_0}^{t_f}L(q,\dot{q})\rd t+\int_{t_0}^{t_f}F(q,\dot{q},u)\delta q\rd t=0.
\end{equation}
Based on \eqref{dale}, the evolution of the system is described in continuous space by the {\it Euler-Lagrange} equations \cite{marsden}
\begin{equation}
\label{lagr}
\frac{\rd}{\rd t}\frac{\partial L}{\partial\dot{q_i}}-\frac{\partial L}{\partial q_i}=F_i,\quad i=1,...,N.
\end{equation}

\textbf{The continuous Pontryagin-d'Alembert principle} To proceed, let  $v=(v_1,...,v_N)^\top\in\mathbb{R}^N$ and $p=(p_1,...,p_N)^\top\in\mathbb{R}^N$ denote the generalized velocity and momentum coordinates respectively. The Pontryagin-d'Alembert principle connects the Lagrangian and Hamiltonian point of views and states that the equations of motion must satisfy
\begin{equation}
\label{ggg9}
\delta\int_{t_0}^{t_f}\big(L(q,v)+p(\dot{q}-v)\big)\rd t+\int_{t_0}^{t_f}F(q,\dot{q},u)\delta q\rd t=0.
\end{equation}
By treating $p$ as a Lagrange multiplier to enforce equality between $v$ and $\dot{q}$, one can immediately see the similarity between \eqref{ggg9} and \eqref{dale}. Manipulation of \eqref{ggg9} yields \cite{varanim}
\begin{equation}
\label{ham}
v_i=\dot{q}_i,\quad p_i=\frac{\partial L}{\partial v_i},\quad\dot{p}_i=\frac{\partial L}{\partial q_i}+F_i,\quad i=1,...,N.
\end{equation}
Eqs. \eqref{ham} are equivalent to \eqref{lagr} and will be used in section \ref{secvigpc} to build the Lagrangian function and non-conservative forces of Polynomial Chaos-based mechanical systems.

\textbf{Discrete Lagrangian Mechanics} Marsden and West in \cite{marsden} showed how to properly calculate the discrete version of \eqref{lagr} and obtain a Variational Integrator (VI). Towards that goal, the discrete Lagrangian $L_d$ is defined by
\begin{equation}
\label{lagrdisc}
L_d(q^k,q^{k+1})\simeq\int_{t_k}^{t_{k+1}} L(q(s),\dot{q}(s))\rd s,
\end{equation}
where $q^k$ denotes the discrete configuration of the system at instant $t_k$ (i.e., $q^k=q(t_k)$). $L_d$ can be estimated via a quadrature rule as
\begin{equation}
\label{lddd}
L_d(q^k,q^{k+1})=L((1-\zeta)q^k+\zeta q^{k+1},\frac{q^{k+1}-q^k}{\Delta t}) \Delta t,
\end{equation}
such that $\Delta t$ is the time step and $\zeta\in[0,1]$. The quadrature is said to be second-order accurate if $\zeta=1/2$. In a similar manner, the continuous non-conservative forces are approximated by left and right discrete forces as follows

\begin{equation}
\begin{split}
\label{fdisc}
&F_d^-(q^k,q^{k+1},u^k)\delta q^k+F_d^+(q^k,q^{k+1},u^k)\delta q^{k+1}\approx\\
&\int_{t_k}^{t_{k+1}}F(q(s),\dot{q}(s),u(s))\delta q \rd s,
\end{split}
\end{equation}
where $u^k$ is the discretization of continuous control inputs. One can typically choose \cite{lin}
\begin{equation}
\label{fddd}
\begin{array}{lcl}
F_d^-(q^k,q^{k+1},u^k)&=&F(\frac{q^k+q^{k+1}}{2},\frac{q^{k+1}-q^k}{\Delta t},u^k) \Delta t,\\
&&\\
F_d^+(q^k,q^{k+1},u^k)&=&0.
\end{array}
\end{equation}
By plugging \eqref{lagrdisc} -- \eqref{fddd} in \eqref{dale}, the {\it Discrete Euler-Lagrange} equations (DEL) can be written as follows
\begin{equation}
\label{del}
\begin{array}{rcl}
0&=&p^k+D_1L_d(q^k,q^{k+1})+F_d^-(q^k,q^{k+1},u^k),\\
&&\\
p^{k+1}&=&D_2L_d(q^k,q^{k+1})+F_d^+(q^k,q^{k+1},u^k).
\end{array}  
\end{equation}
In the expressions above, $D_iL_d(\cdot)$ denotes the partial derivative of $L_d$ with respect to its $i^{\text{th}}$ argument, while $p^k$ can be viewed as the discrete momentum at time $t_k$. The DEL equations can now be used to propagate the states of the system forward in time. Specifically, given the current state, $(q^k,p^k)$, the next state is determined by solving \eqref{del}. 


It has been shown that variational integration methods outperform standard integration schemes (e.g., Euler integration, Runge Kutta, etc) \cite{marsden} that discretize the equations of motion instead of the Lagrangian function. Specifically, by propagating the DEL equations, we expect to get a behavior similar to the continuous dynamic equations. In fact, VI's exhibit a structure preserving nature, as well as improved performance in terms of accuracy and energy stability.

\section{Polynomial Chaos representations of dynamical systems}\label{secvigpc}
This section uses Polynomial Chaos theory to derive expressions for the dynamics of stochastic systems. First, the evolution of the Polynomial Chaos expansion coefficients is described. The time-varying coefficients capture the probabilistic evolution of the system as it is influenced by its internal dynamics and external inputs. Next, we develop a Variational Integrator that allows us to propagate the discrete-time gPC formulation efficiently.

\subsection{Evolution of Polynomial Chaos-based dynamical systems}
Consider the case where a dynamical system is influenced by a set of uncertain internal parameters. We represent these parameters by a collection of independent random variables $\lambda^p\in\mathbb{R}^{d_p}$ with known distribution\footnote{We use superscripts $p$ 
and $0$ to associate random variables $\lambda$ and $\xi$ with uncertainty in the parameters 
and initial states respectively. Not to be confused with the notation $\cdot^k$ that indicates evaluation of the term $\cdot$ at $t_k$ ($k=0,1,...$).}. Similarly, the initial state will be stochastic and will depend on the independent random variables $\lambda^0\in\mathbb{R}^{d_0}$. The dynamics equations are
\begin{equation}
\label{f1}
\dot{x}(t)=f(x(t),u(t),t;\lambda^p),\quad x^0=x(t_0;\lambda^0).
\end{equation}
In this paper, the controls are considered to be deterministic. We also assume that all random quantities belong in their associated $L_\rho^2$ spaces (see eq. \eqref{Lro}).

The first step is to take the gPC expansion for the random parameters $\lambda^p$, $x^0$ with respect to a set of standard random variables $\xi^p$, $\xi^0$. In this way, the density function of the new random variables accords with the weight functions presented in Table \ref{tab:ask}. Hence, we have
\begin{equation}
\label{gpcl}
\lambda_i^p\approx\sum_{j=0}^{\Lambda}\lambda_{ij}^p\phi^p_j(\xi^p),\quad i=1,...,d_p.
\end{equation}
By using a similar expression for $\lambda^0$, the initial states are expanded as
\begin{equation}
\label{gpcx0}
x_i(t_0)\approx\sum_{j=0}^{K_0}x_{ij}(t_0)\phi^0_j(\xi^0),\quad i=1,...,n.
\end{equation}
Now define the concatenated vector of random inputs as $\xi=(\xi^p,
\xi^0)^\top\in\mathbb{R}^d$, where $d=d_p
+d_0$. Since the evolution of $x$ depends on $\xi$, the associated gPC expansion of the state vector is defined as
\begin{equation}
\label{gpcx}
x_i(t)\approx\sum_{j=0}^K x_{ij}(t)\phi_j(\xi),\quad i=1,...,n.
\end{equation}
Differentiating \eqref{gpcx} gives
\begin{equation}
\label{gpcxdot}
\dot{x}_i(t)\approx\sum_{j=0}^K \dot{x}_{ij}(t)\phi_j(\xi).
\end{equation}
Note that each set of polynomials $\phi$ in \eqref{gpcl} -- \eqref{gpcx} has to be orthogonal with respect to the corresponding weight functions. After plugging eqs. \eqref{gpcl} to \eqref{gpcxdot} in \eqref{f1}, Galerkin projection can be performed as shown in \eqref{galerkin}. One can obtain
\begin{equation}
\label{ccc}
\dot{x}_{ij}(t)=\frac{\int_\mathcal{D} f_i(x(t,\xi),u(t),t,\xi)\phi_j(\xi)\rho(\xi)\rd\xi}{\langle\phi_j,\phi_j\rangle},
\end{equation}
where $\rho(\cdot)$, $\mathcal{D}$ denote the probability density function and domain of $\xi$, respectively. The evolution of the gPC coefficients can thus be compactly written as
\begin{equation}
\label{gpcX}
\dot{\bold{X}}(t)=\bold f(\bold{X}(t),u(t),t),
\end{equation}
where $\bold{X}=(x_{10},...,x_{1K},...,x_{ij},...,x_{nK})^\top\in\mathbb{R}^{n(K+1)}$ 
and the elements of $\bold f$ are given by \eqref{ccc}.

A major benefit of gPC theory is that the moments of the expanded stochastic process can be estimated analytically \cite{bookxiu}. More precisely, by using the orthogonality of $\{\phi_j\}$ and recalling that $\phi_0(\xi)=1$ for all types of polynomials in Table \ref{tab:ask}, one obtains
\begin{flalign}
\label{gggg_exp}
\mathbb{E}[x_i(t)]\approx&(\hat{\mathcal{M}}_i)_1=\nonumber\\
&\int_\mathcal{D}\sum_{j=0}^K x_{ij}(t)\phi_j(\xi)\rho(\xi)\rd\xi=x_{i0}(t),&
\end{flalign}
\begin{flalign}
\label{gggg_var}
\text{var}[x_i(t)]\approx&(\hat{\mathcal{M}}_i)_2=\nonumber\\
\int_\mathcal{D}\big(&\sum_{j=0}^K x_{ij}(t)\phi_j(\xi)\big)\big(\sum_{l=0}^K x_{il}(t)\phi_l(\xi)\big)\rho(\xi)\rd\xi\nonumber\\
&\quad-\bigg(\int_\mathcal{D}\sum_{j=0}^K x_{ij}(t)\phi_j(\xi)\rho(\xi)\rd\xi\bigg)^2=\nonumber\\
&\sum_{j=1}^Kx_{ij}^2(t)\langle \phi_j,\phi_j\rangle,&
\end{flalign}
\begin{flalign}
\label{gggg_skew}
&\text{skew}[x_i(t)]\approx(\hat{\mathcal{M}}_i)_3=\nonumber\\
&\sum_{j=0}^K\sum_{g=0}^Kx_{ij}(t)x_{ig}(t)\big(\sum_{l=0}^Kx_{il}(t)\langle\phi_j,\phi_g,\phi_l\rangle-3x_{i0}(t)\big)+2x_{i0}^3(t),\\
&\quad\vdots\nonumber&
\end{flalign}
\begin{flalign}
\label{Mij}
&(\mathcal{M}_i(x))_j\approx(\hat{\mathcal{M}}_i(\textbf X))_j,&
\end{flalign}
where we have used $(\hat{\mathcal{M}}_i)_j$ to denote the gPC estimate of the $j^\text{th}$ (central) moment of state $x_i$. Remarkably, propagating the set of deterministic ordinary differential equations in \eqref{gpcX} results in obtaining an estimate of the state distribution over time. This further implies that the probabilistic evolution of the system can be influenced by controlling \eqref{gpcX}.
\begin{remark}
	As explained in section \ref{secgpc}, Polynomial Chaos expansions constitute a form of orthogonal projection. When the quantity to be approximated belongs in $L^2_\rho$, convergence to its true value is guaranteed \cite{xiu}. Hence, the states in \eqref{gpcX} satisfy
	\begin{equation*}
	||x_i(\xi)-\sum_{j=0}^Kx_{ij}\phi_j(\xi)||_\rho\rightarrow0,\quad as\quad K\rightarrow\infty,
	\end{equation*}
	with $||\cdot||_\rho$ being the corresponding norm in $L^2_\rho$.
\end{remark}

\subsection{Disrete Euler-Lagrange equations for gPC expasions of mechanical systems}
Designing faithful discrete representations for continuous equations of motion constitutes a key ingredient in discrete-time optimal control methods. In our case, one could discretize eq. \eqref{gpcX} directly. However, as explained in section \ref{secvi}, such a naive approach usually induces numerical errors during simulation.

Suppose a mechanical system satisfies eqs. \eqref{ham} and is influenced by uncertainty in the parameters or initial states. Then, we can take the gPC expansion of $q$, $v$ and $p$ with respect to the random inputs $\xi\in\mathbb{R}^d$ as
\begin{equation}
\begin{split}
\label{gpchamil}
q_i(t,\xi)&\approx\sum_{j=0}^K q_{ij}(t)\phi_j(\xi),\quad v_i(t,\xi)\approx\sum_{j=0}^K v_{ij}(t)\phi_j(\xi)\\
p_i(t,\xi)&\approx\sum_{j=0}^K p_{ij}(t)\phi_j(\xi),
\end{split}
\end{equation}
where $i=1,...,N$. Next, define the concatenated vectors $\bold{Q}=(q_{10},...,q_{1K},...,q_{NK})\in\mathbb{R}^{N(K+1)}$ and $\bold{V}=(v_{10},...,v_{1K},...,v_{NK})\in\mathbb{R}^{N(K+1)}$. Let also $\bold{\hat{P}}=(\hat p_{10},...,\hat p_{1K},...,\hat{p}_{NK})\in\mathbb{R}^{N(K+1)}$ be the set of unnormalized momentum coefficients with $\{\hat{p}_{ij}\}=\{p_{ij}\langle\phi_j,\phi_j\rangle\}$.
The authors in \cite{pasini} showed that for uncertain conservative systems, the coefficients of the gPC expansions satisfy Hamilton's equations. Here, we extend this concept to show how the DEL equations can be transformed, especially when non-conservative forces are applied. Consequently, a variational integration scheme is designed that can be used for propagating the density of stochastic mechanical systems.
\begin{lemma}
\label{theoremgpcdel}
	Consider a mechanical system with uncertain parameters and let $L$, $F$ denote the Lagrangian function and set of non-conservative forces respectively. Suppose also that its position, velocity and momentum coordinates can be expanded as in \eqref{gpchamil}. Then, the system of the gPC coefficients $(\emph{\textbf Q},\emph{\textbf V},\hat{\emph{\textbf P}})$ will satisfy eqs. \eqref{ham} with
	
	\begin{equation}
	\label{Lhat}
	\hat{L}(\bold{Q},\bold{V})=\int_\mathcal{D}L\rho(\xi)\rd\xi
	\end{equation}
	being the associated  Lagrangian function and
	
	\begin{equation}
	\label{gpcf}
	\begin{split}
	&\bold{\hat{F}}=(\hat{F}_{10},...,\hat{F}_{1K},...,\hat{F}_{NK})\in\mathbb{R}^{N(K+1)},\\
	&\hat{F}_{ij}(\bold{Q},\bold{V},u)=\int_\mathcal{D}F_i\phi_j(\xi)\rho(\xi)\rd\xi
	\end{split}
	\end{equation}
	being the non-conservative forces.
\end{lemma}
\begin{proof}
	For Lagrangian systems, eqs. \eqref{ham} must be satisfied. Plugging \eqref{gpchamil} in \eqref{ham} and performing Galerkin projection gives
	
	\begin{equation}
	\label{gpcham1}
	v_{ij}=\dot{q}_{ij},
	\end{equation}
	\begin{equation}
	\label{gpcham3}
	\hat p_{ij}=\int_\mathcal{D}\frac{\partial L}{\partial v_i}\phi_j(\xi)\rho(\xi)\rd\xi,
	\end{equation}
	\begin{equation}
	\label{gpcham4}
	\dot{\hat p}_{ij}=\int_\mathcal{D}\frac{\partial L}{\partial q_i}\phi_j(\xi)\rho(\xi)\rd\xi+\int_\mathcal{D} F_i\phi_j(\xi)\rho(\xi)\rd\xi,
	\end{equation}
	for $i=1,...,N$ and $j=0,...,K$. Now define the functions $\hat{L}$, $\hat{F}_{ij}$ as in \eqref{Lhat}, \eqref{gpcf} respectively. Since the gPC coefficients are deterministic, one can obtain
	
	\begin{equation}
	\label{th1proof_eq}
	\frac{\partial\hat{L}}{\partial v_{ij}}=\int_\mathcal{D}\frac{\partial L}{\partial v_i}\frac{\partial v_i}{\partial v_{ij}}\rho(\xi)\rd\xi=\int_\mathcal{D}\frac{\partial L}{\partial v_i}\phi_j(\xi)\rho(\xi)\rd\xi,
	\end{equation}
	where the first equality is due to chain rule and the second equality is due to \eqref{gpchamil}. Eqs. \eqref{gpcham3} and \eqref{th1proof_eq} yield
	\begin{equation}
	\label{ggg1}
	\hat p_{ij}=\frac{\partial\hat{L}}{\partial v_{ij}}.
	\end{equation}
	In a similar manner, one can show that
	
	\begin{equation}
	\label{th1proof_eq2}
	\frac{\partial\hat{L}}{\partial q_{ij}}=\int_\mathcal{D}\frac{\partial L}{\partial q_i}\frac{\partial q_i}{\partial q_{ij}}\rho(\xi)\rd\xi=\int_\mathcal{D}\frac{\partial L}{\partial q_i}\phi_j(\xi)\rho(\xi)\rd\xi.
	\end{equation}
	Combining \eqref{gpcf}, \eqref{gpcham4} and \eqref{th1proof_eq2} gives
	
	\begin{equation}
	\label{ggg7}
	\dot{\hat p}_{ij}=\frac{\partial\hat{L}}{\partial q_{ij}}+\hat{F}_{ij}.
	\end{equation}
	For clarity, we put eqs. \eqref{gpcham1}, \eqref{ggg1} and \eqref{ggg7} together to get
	
	\begin{equation}
	\label{gpchampon}
	v_{ij}=\dot{q}_{ij},\quad\hat p_{ij}=\frac{\partial\hat{L}}{\partial v_{ij}},\quad\dot{\hat p}_{ij}=\frac{\partial\hat{L}}{\partial q_{ij}}+\hat{F}_{ij}.
	\end{equation}
	By comparing eqs. \eqref{gpchampon} with \eqref{ham}, the conclusion is made.
\end{proof}

Lemma \ref{theoremgpcdel} implies that the set of gPC-based coordinates in \eqref{gpchamil} behaves as a Lagrangian system. In this context, $\hat{L}$ is the Lagrangian, $\bold{\hat{F}}$ are the non-conservative forces and $\bold{Q}$, $\bold{V}$, $\bold{\hat P}$ denote the position, velocity and (unnormalized) momentum coordinates respectively. Since principles \eqref{dale} and \eqref{ggg9} are equivalent, the corresponding {\it Euler-Lagrange} equations will be satisfied in the continuous domain. As a consequence, one can define the DEL equations for the gPC representation of a mechanical system as
\begin{equation}
\label{delg1}
0=\hat p_{ij}^k+\displaystyle{\frac{\partial}{\partial q_{ij}^k}}\hat{L}_d(\bold{Q}^k,\bold{Q}^{k+1})+\hat{F}_{dij}^-(\bold{Q}^k,\bold{Q}^{k+1},u^k),
\end{equation}

\begin{equation}
\label{delg2}
\hat p_{ij}^{k+1}=\displaystyle{\frac{\partial}{\partial q_{ij}^{k+1}}}\hat{L}_d(\bold{Q}^k,\bold{Q}^{k+1})+\hat{F}_{dij}^+(\bold{Q}^k,\bold{Q}^{k+1},u^k),
\end{equation}
where $i=1,...,N$ and $j=0,...,K$. The discrete versions of the Lagrangian $\hat{L}_d$ and non-conservative forces $\hat{F}_{dij}^\pm$ may be computed as in \eqref{lddd}, \eqref{fddd} respectively. The required steps for propagating the discrete gPC representation of dynamical systems are presented in Algorithm \ref{algdel}.

\begin{algorithm}[h]
	\SetAlgoLined
	\KwData{Lagrangian $L$, non-conservative forces $F$, horizon $K_f$, $(q^0,p^0)$;}
	Determine $\{(q^0_{ij},\hat p^0_{ij})\}$ by performing Galerkin projection on \eqref{gpchamil};\\
	Compute $\hat{L}_d$, $\{\hat{F}_{dij}^\pm\}$ by discretizing eqs. \eqref{Lhat}, \eqref{gpcf} respectively (e.g., by using \eqref{lddd}, \eqref{fddd});\\
	\For{k=0:$K_f$-1}{
		Given $\{(q^k_{ij},\hat p^k_{ij})\}$, compute $\{q^{k+1}_{ij}\}$ by implicitly solving \eqref{delg1};\\
Get $\{\hat p^{k+1}_{ij}\}$ by directly solving \eqref{delg2};
}
	\caption{Propagation of DEL equations for gPC representations of mechanical systems}
	\label{algdel}
\end{algorithm}

\section{Optimal control of Polynomial Chaos-based dynamical systems}\label{secgpcddp}
In this section, we develop a numerical scheme for obtaining (locally) optimal trajectories for systems with parametric uncertainties. Our approach constitutes an extension of the original Differential Dynamic Programming (DDP) algorithm. DDP's main features are its fast convergence rates and scalability, attained by sacrificing global optimality \cite{ddp}. The aforementioned advantages have allowed researchers to apply this methodology on non-linear, high-dimensional systems \cite{conlddp}, \cite{recd}.

We begin by formulating the gPC analogue of the moment-based stochastic optimal control problem in \eqref{smxi}. Moreover, we show that expected costs with quadratic terms (usually found in standard stochastic optimal control theory) can be viewed as a subclass of our gPC formulation. We then develop our trajectory-optimization framework, gPC-DDP, for solving the obtained type of problems. Certain modifications are also provided in order to incorporate the Variational Integrator developed in section IV. Lastly, we prove that under some mild conditions, gPC-DDP converges globally to a stationary solution, with the convergence rate being locally quadratic.

\subsection{gPC formulation of stochastic optimal control problems under parametric uncertainties}
Let $x\in\mathbb{R}^n$, $u\in\mathbb{R}^m$ denote the state and control input vectors respectively. Consider the stochastic optimal control problem in \eqref{smxi}, which is restated below for convenience.

\begin{equation}
\label{smxi0134}
\begin{split}
&\min_u\hspace{1.3mm} \int_{t_0}^{t_f}\mathcal{L}(\mathcal{M}(t),u,t)\rd t+\mathcal{F}(\mathcal{M}(t_f),t_f)\\
\text{s.t.}&\quad\dot{x}(t)=f(x(t),u(t),t;\lambda^p),\quad x^0=x(t_0;\lambda^0).
\end{split}
\end{equation}
Based on the analysis of section \ref{secvigpc}, eqs. \eqref{gpcX} -- \eqref{Mij} allow us to transform \eqref{smxi0134} into the following deterministic optimal control problem

\begin{equation}
\label{gpcstochoptprob1}
\begin{split}
&\min_u\hspace{1.7mm}\int_{t_0}^{t_f} \bm{L}(\textbf{X},u,t)\rd t+\bm{F}(\textbf{X}(t_f),t_f)\\
&\text{s.t.}\quad\dot{\bold{X}}(t)=\bold f(\bold{X},u,t),\quad \bold{X}^0 = \overline{\bold{X}}(t_0),
\end{split}
\end{equation}
with $\bm{L}(\textbf{X},u,t)=\mathcal{L}(\hat{\mathcal{M}}(\textbf X(t)),u,t)$ and $\bm{F}(\textbf{X}(t_f))=\mathcal{F}(\hat{\mathcal{M}}(\textbf X(t_f)),t_f)$. The new state vector will include the gPC coefficients of $x$ (i.e., $\bold{X}=(x_{10},...,x_{1K},...,x_{ij},...,x_{nK})^\top\in\mathbb{R}^{n(K+1)}$) and the dynamic constraints are derived as shown in \eqref{gpcX}. In addition, note that since the statistics of $\xi^0$ in \eqref{gpcx0} are assumed to be known, the initial state $\textbf X^0$ will be fixed.
\begin{remark}
	The deterministic optimal control problem in \eqref{gpcstochoptprob1} is an approximation to the stochastic optimal control problem in \eqref{smxi0134}. The two formulations become equivalent if the orthogonal projection of the state in \eqref{gpcx} holds with equality for all time instants.
\end{remark}

Now let us consider the problem of minimizing an expected cost of the following form
\begin{equation}
\label{Jexp}
\min_u\hspace{1.3mm}\mathbb{E}\bigg[\int_{t_0}^{t_f} \mathcal{L}_e(x,u)\rd t+\mathcal{F}_e(x(t_f))\bigg].
\end{equation}
 These types of cost functions are encountered in standard stochastic optimal control theory \cite{bookstochastic}. Let also $\mathcal{L}_e$ and $\mathcal{F}_e$ be quadratic functions defined as
\begin{equation}
\label{runnL}
\mathcal{L}_e(x,u)=\frac{1}{2}\big[(x(t)-x^{goal}(t))^\top S(x(t)-x^{goal}(t))+u^\top(t)Ru(t)\big],
\end{equation}
\begin{equation}
\label{termF}
\mathcal{F}_e(x(t_f))=\frac{1}{2}(x(t_f)-x^{goal}(t_f))^\top S_f(x(t_f)-x^{goal}(t_f)),
\end{equation}
where $S$ and $S_f$ are diagonal, positive semi-definite matrices, while $R$ is a positive definite matrix. Below, we show that this problem can be viewed as a subclass of \eqref{gpcstochoptprob1}. A similar derivation was provided in \cite{fisher}. However, here we also include a desired state $x^{goal}$ at each time instant, which is useful for tracking tasks. In addition, we consider state weighting matrices without cross terms.

\begin{lemma}
	The expected cost in \eqref{Jexp} with quadratic terms as in \eqref{runnL}, \eqref{termF}, can be transformed into the deterministic cost function of \eqref{gpcstochoptprob1} with
	\begin{equation}
		\label{gpcrunnL}
		\bm{L}=\frac{1}{2}(\textbf{X}(t)-\textbf{X}^{goal}(t))^\top \bm{S}(\textbf{X}(t)-\textbf{X}^{goal}(t))+\frac{1}{2}u^\top(t)Ru(t),
		\end{equation}
		\begin{equation}
		\label{gpctermF}
		\bm{F}(\textbf{X}(t_f))=\frac{1}{2}(\textbf{X}(t_f)-\textbf{X}^{goal}(t_f))^\top \bm{S_f}(\textbf{X}(t_f)-\textbf{X}^{goal}(t_f)),
		\end{equation}
		where
		\begin{equation}
		\label{gpc_S}
		\bm{S}=S\otimes \emph{diag}(\langle\phi_0,\phi_0\rangle,\langle\phi_1,\phi_1\rangle,...,\langle\phi_K,\phi_K\rangle),
		\end{equation}
		\begin{equation}
		\label{gpc_Sf}
		\bm{S_f}=S_f\otimes \emph{diag}(\langle\phi_0,\phi_0\rangle,\langle\phi_1,\phi_1\rangle,...,\langle\phi_K,\phi_K\rangle),
		\end{equation}
		\begin{equation}
		\label{gpc_Xg}
		\textbf{X}^{goal}(t_f)=(x_1^{goal}(t_f),\textbf{0}_{1\times K},...,x_n^{goal}(t_f),\textbf{0}_{1\times K})^\top\in\mathbb{R}^{n(K+1)}
		\end{equation}
		and $\otimes$ is the Kronecker product.
\end{lemma}
\begin{proof}
    We will consider only the running cost since a similar analysis can be performed for $\mathcal{F}_e$. One has
    \begin{equation}
    \begin{split}
    \label{th2_eq1}
    &\mathbb{E}[\mathcal{L}_e(x,u)]=\\
    	&\mathbb{E}[\frac{1}{2}(x(t)-x^{goal}(t))^\top S(x(t)-x^{goal}(t))]+\frac{1}{2}u^\top(t)Ru(t)=\\
    	&\frac{1}{2}\mathbb{E}\bigg[x(t)^\top Sx(t)-2x(t)^\top Sx^{goal}(t)+x^{goal}(t)^\top Sx^{goal}(t)+\\
    	&\frac{1}{2}u^\top(t)Ru(t)\bigg]=\\
    	&\frac{1}{2}\bigg(\text{trace}(\mathbb{E}[x(t)x(t)^\top]S)-2\mathbb{E}[x(t)]^\top Sx^{goal}(t)+\\
    	&x^{goal}(t)^\top Sx^{goal}(t)\bigg)+\frac{1}{2}u^\top(t)Ru(t).
    	\end{split}
    \end{equation}
    Now note that
    \begin{equation}
    \label{th2_eq2}
    \text{trace}(\mathbb{E}[x(t)x(t)^\top]S)=\text{trace}(\text{cov}[x(t)]S)+\mathbb{E}[x(t)]^\top S\mathbb{E}[x(t)],
    \end{equation}
    where $\text{cov}[\cdot]$ denotes the covariance operator. Plugging \eqref{th2_eq2} in \eqref{th2_eq1} gives
    \begin{equation}
    \label{eq_L}
    \begin{split}
    &\mathbb{E}[\mathcal{L}_e(x,u)]=\\
    &\frac{1}{2}(\mathbb{E}[x(t)]-x^{goal}(t))^\top S(\mathbb{E}[x(t)]-x^{goal}(t))+\text{trace}(\text{cov}[x(t)]S)\\
    &+\frac{1}{2}u^\top(t)Ru(t)\approx\\
    &\frac{1}{2}\big((x_{10},...,x_{n0})^\top-x^{goal}(t)\big)^\top S\big((x_{10},...,x_{n0})^\top-x^{goal}(t)\big)+\\
    &\text{trace}(\text{cov}[x(t)]S)+\frac{1}{2}u^\top(t)Ru(t),
    \end{split}
    \end{equation}
    where the last equality is due to \eqref{gggg_exp}. Now by using \eqref{gggg_var} and recalling that $S$ is diagonal, the second term of \eqref{eq_L} becomes
    \begin{equation}
    \begin{split}
    \label{th2_eq3}
    &\text{trace}(\text{cov}[x(t)]S)=\sum_{i=1}^n\text{var}[x_i(t)]S_{ii}\approx\sum_{i=1}^n\sum_{j=1}^Kx_{ij}^2(t)\langle\phi_j,\phi_j\rangle S_{ii}=\\
    &\sum_{i=1}^n(x_{i1},...,x_{iK})\begin{bmatrix}
    S_{ii}\langle\phi_1,\phi_1\rangle&\dots&0\\
    \vdots&\ddots&\vdots\\
    0&\dots&S_{ii}\langle\phi_K,\phi_K\rangle
    \end{bmatrix}\begin{pmatrix}x_{i1}\\\vdots\\x_{iK}\end{pmatrix}.
    \end{split}
    \end{equation}
    After plugging \eqref{th2_eq3} into \eqref{eq_L} and noting that $\langle\phi_0,\phi_0\rangle=1$ for all polynomials in Table \ref{tab:ask}, equation \eqref{gpcrunnL} is obtained.
\end{proof}

By observing eqs. \eqref{gpc_S} -- \eqref{gpc_Xg} $\&$ \eqref{eq_L}, \eqref{th2_eq3}, we make the following remarks.
\begin{remark}
	The optimal control problem in \eqref{gpcstochoptprob1} with cost terms given in \eqref{gpcrunnL} -- \eqref{gpc_Xg}, penalizes trajectories with: i) large deviation between the expected states and the target states, ii) high variance.
\end{remark}

\begin{remark}
	The weighting matrices in \eqref{gpc_S}, \eqref{gpc_Sf} affect all gPC coefficients of a particular state equivalently. In practice, assigning different weights can provide greater freedom between penalizing large expected errors and high variance. For example, one can define $\bm{S_f}$ as  
	
	\begin{equation}
	\label{gpc_Sf0}
	\bm{S_f}=
	\begin{bmatrix}
	\bm{S_{f1}}&\dots&\bold{0}_{(K+1)}\\
	\vdots&\ddots&\vdots\\
	\bold{0}_{(K+1)}&\dots&\bm{S_{fn}}
	\end{bmatrix},
	\end{equation}
	with
	\begin{equation}
	\label{gpcsf0001}
	\bm{S_{fi}}=\text{diag}(s_{f_{i0}},s_{f_{i1}}\langle\phi_1,\phi_1\rangle ,\dots,s_{f_{iK}}\langle\phi_K,\phi_K\rangle )\in\mathbb{R}^{(K+1)\times(K+1)}.
	\end{equation}
	The running cost, $\bm{L}$ may be defined similarly.
\end{remark}
In our simulated examples, we will consider quadratic cost functions with weighting matrices as in \eqref{gpc_Sf0}, \eqref{gpcsf0001}.

\subsection{The gPC-DDP framework in discrete time}
The Differential Dynamic Programming algorithm has been developed in \cite{ddp} for the deterministic case. It numerically solves the optimal control problem by using quadratic expansions of the dynamics and the cost function along nominal trajectories. The scheme is iterative in nature such that it computes the optimal control deviation given a nominal input signal. The nominal input is then updated and the process is repeated until convergence.

We extend the original DDP method in order to handle Polynomial Chaos representations of the dynamics and the cost. The goal is to solve the generic optimal control problem
\begin{equation}
\label{gpcstochoptprob3}
\begin{split}
&\min_u\hspace{1.7mm}\bm J(\textbf X^0,u)\\
\text{s.t.}\quad\dot{\bold{X}}(t)&=\bold f(\bold{X},u,t),\quad \bold{X}^0 = \overline{\bold{X}}(t_0),
\end{split}
\end{equation}
with
\begin{equation}
\label{Jgpc}
\bm J(\textbf X^0,u)=\int_{t_0}^{t_f} \bm{L}(\textbf{X},u,t)\rd t+\bm{F}(\textbf{X}(t_f),t_f).
\end{equation}
We start by defining the value function as follows

\begin{equation}
\label{v}
V(\textbf X,t)=\displaystyle\min_u\bm J(\textbf{X}(t),u),
\end{equation}
In the discrete-time setting, the derivation is based on Bellman's principle of optimality. One can write \cite{ddp}
\begin{equation}
\label{v17}
V(\textbf{X}(t_k),t_k)=\min_u\bigg[\int_{t_k}^{t_{k+1}} \bm{L}(\textbf{X},u)\rd t+V(\textbf{X}(t_{k+1}),t_{k+1})\bigg],
\end{equation}
where $t_k$ denotes the $k^{\text{th}}$ time instant (i.e., $t_k=k\Delta t$, such that $\Delta t$ is the time step). The integral on the right hand side of \eqref{v17} will be approximated via a left hand rectangle method. In what follows, define the $Q$-function as
\begin{equation}
\label{Q}
Q(\textbf{X}(t_k),u(t_k))=\bm{L}^k+V(\textbf{X}(t_{k+1}),t_{k+1}),
\end{equation}
with $\bm{L}^k=\bm{L}(\textbf{X}(t_{k}),u(t_{k}))\Delta t$. Assume now that a nominal control input $\bar{u}$ and the associated state trajectory $\overline{\textbf{X}}$ are given. The first step will be to linearize \eqref{gpcX} about $\overline{\textbf{X}}$, $\bar{u}$. Suppose that that the dynamics and cost functions are differentiable up to the second order. Using an Euler discretization scheme gives
\begin{equation}
\label{deltax}
\begin{split}
&\delta\textbf{X}(t_{k+1})\approx\Theta(t_k)\delta \textbf{X}(t_k)+B(t_k)\delta u(t_k)+\frac{1}{2}\Delta t\times\\
&\bigg[\sum_{i,j=1}^{\bm n}\nabla_{\textbf{x}_i\textbf{x}_j}\textbf{f}(t_k)\delta\textbf{X}_i(t_k)\delta\textbf{X}_j(t_k)+\sum_{i,j=1}^m\nabla_{u_iu_j}\textbf{f}(t_k)\delta u_i(t_k)\delta u_j(t_k)\\
&+
2\sum_{i,j=1}^{\textbf n,m}\nabla_{\textbf{x}_iu_j}\textbf{f}(t_k)\delta\textbf{X}_i(t_k)\delta u_j(t_k)\bigg],
\end{split}
\end{equation}
where $\bm n=n(K+1)$ is the dimension of the gPC-based system (i.e., $\textbf{X}\in\mathbb{R}^{\bm n}$) and $\Theta(t_k)=(I+\Delta t\nabla_\textbf{x}\textbf{f}(\textbf{X},u,t_k))|_{\bar{\textbf{x}},\bar{u}}$, $B(t_k)=\Delta t\nabla_u\textbf{f}(\textbf{X},u,t_k)|_{\bar{\textbf{x}},\bar{u}}$. Note that all terms on the right hand side of \eqref{deltax} are evaluated at the nominal trajectories. Moreover, $\delta \textbf{X}(t_k)$, $\delta u(t_k)$ represent state and control deviations about $\overline{\textbf{X}}(t_k)$ and $\bar{u}(t_k)$ respectively. 

To proceed, eq. \eqref{Q} will be quadratically approximated. By using \eqref{deltax}, one can obtain
\begin{equation}
\label{Qexp}
\begin{split}
&Q(\textbf{X}(t_k),u(t_k))\approx Q_0^k+\delta \textbf{X}(t_{k})^\top Q_{\textbf{x}}^k+\delta u(t_{k})^\top Q_{u}^k\\
&+\frac{1}{2}\delta \textbf{X}(t_{k})^\top Q_{\textbf{xx}}^k\delta \textbf{X}(t_{k})+\frac{1}{2}\delta u(t_{k})^\top Q_{uu}^k\delta u(t_{k})\\
&+\frac{1}{2}\delta u(t_{k})^\top Q_{u\textbf{x}}^k\delta \textbf{X}(t_{k})+\frac{1}{2}\delta \textbf{X}(t_{k})^\top Q_{\textbf{x}u}^k\delta u(t_{k}),
\end{split}
\end{equation}
where
\begin{equation}
\label{Qddp}
\begin{split}
Q_0^k&=\bm{L}^k|_{\bar{\textbf{x}},\bar{u}}+V^{k+1}|_{\bar{\textbf{x}}},\\
Q_x^k&=\bm{L}^k_\textbf{x}|_{\bar{\textbf{x}},\bar{u}}+(\Theta^k)^\top V_{\textbf{x}}^{k+1}|_{\bar{\textbf{x}}},\\
Q_u^k&=\bm{L}^k_u|_{\bar{\textbf{x}},\bar{u}}+(B^k)^\top V_{\textbf{x}}^{k+1}|_{\bar{\textbf{x}}},\\
Q_{\textbf{xx}}^k&=\bm{L}^k_{\textbf{xx}}|_{\bar{\textbf{x}},\bar{u}}+(\Theta^k)^\top V_{\textbf{xx}}^{k+1}|_{\bar{\textbf{x}}}\Theta^k+(\Delta t\nabla_{\textbf{x}\textbf{x}}\textbf{f}^k|_{\bar{\textbf{x}},\bar{u}})\times_1V_{\textbf{x}}^{k+1}|_{\bar{\textbf{x}}},\\
Q_{\textbf{x}u}^k&=\bm{L}^k_{\textbf{x}u}|_{\bar{\textbf{x}},\bar{u}}+(\Theta^k)^\top V_{\textbf{xx}}^{k+1}|_{\bar{\textbf{x}}}B^k+(\Delta t\nabla_{\textbf{x}u}\textbf{f}^k|_{\bar{\textbf{x}},\bar{u}})\times_1V_{\textbf{x}}^{k+1}|_{\bar{\textbf{x}}},\\
Q_{u\textbf{x}}^k&=\bm{L}^k_{u\textbf{x}}|_{\bar{\textbf{x}},\bar{u}}+(B^k)^\top V_{\textbf{xx}}^{k+1}|_{\bar{\textbf{x}}}\Theta^k+(\Delta t\nabla_{u\textbf{x}}\textbf{f}^k|_{\bar{\textbf{x}},\bar{u}})\times_1V_{\textbf{x}}^{k+1}|_{\bar{\textbf{x}}},\\
Q_{uu}^k&=\bm{L}^k_{uu}|_{\bar{\textbf{x}},\bar{u}}+(B^k)^\top V_{\textbf{xx}}^{k+1}|_{\bar{\textbf{x}}}B^k+(\Delta t\nabla_{uu}\textbf{f}^k|_{\bar{\textbf{x}},\bar{u}})\times_1V_{\textbf{x}}^{k+1}|_{\bar{\textbf{x}}}
\end{split}
\end{equation}
and $\nabla_{\textbf{x}\textbf{x}}\textbf{f}\in\mathbb{R}^{\bm n\times\bm n\times\bm n}$, $\nabla_{\textbf{x}u}\textbf{f}\in\mathbb{R}^{\bm n\times\bm n\times m}$, $\nabla_{u\textbf{x}}\textbf{f}\in\mathbb{R}^{\bm n\times m\times\bm n}$, $\nabla_{uu}\textbf{f}\in\mathbb{R}^{\bm n\times m\times m}$ are three-dimensional tensors. We denote by $\times_i$ the ``mode-$i$" multiplication between a tensor and a vector (or matrix) \cite{matlabtensor}. For example, we have that $(\nabla_{\textbf{x}\textbf{x}}\textbf{f}^k|_{\bar{\textbf{x}},\bar{u}})\times_1V_{\textbf{x}}^{k+1}|_{\bar{\textbf{x}}}=\sum_{i=1}^{\bm n}[V_{\textbf{x}}(t_{k+1})|_{\bar{\textbf{x}}}]_i\nabla_{\textbf{x}\textbf{x}}[\textbf{f}(t_k)|_{\bar{\textbf{x}},\bar{u}}]_i$. The remaining multiplications are computed similarly\footnote{$[\textbf f]_i$ denotes the $i^{\text{th}}$ element of $\textbf f$.}.

Now, in order for the right hand side of \eqref{v17} to attain its minimum, its derivative with respect to $\delta u$ must be equal to zero. By plugging \eqref{Q}, \eqref{Qexp} and \eqref{Qddp} into \eqref{v17} and following this strategy, one can get the optimal control deviation
\begin{equation}
\label{d13}
\delta u^*(t_k)=\ell(t_k)+\Sigma(t_k)\delta \textbf{X}(t_k),
\end{equation}
where
\begin{equation*}
\ell(t_k)=-Q_{uu}^{-1}(t_k)Q_u(t_k),\quad \Sigma(t_k)=-Q_{uu}^{-1}(t_k)Q_{u\textbf x}(t_k).
\end{equation*}
We then compute the updated control trajectory as follows
\begin{equation}
\label{newu}
u(t_k)=\bar{u}(t_k)+\gamma\ell(t_k)+\Sigma(t_k)\delta\textbf{X}(t_k),\quad k=0,1,...
\end{equation}
The term $\gamma$ satisfies $0<\gamma\leq1$. It is introduced for performing line searches, so that convergence is attained even for nonlinear problems (see section \ref{secconvg}). Specifically, it is initialized at each iteration as $\gamma=1$ and is reduced until cost reduction is attained.

Note that the controls update in \eqref{newu}, requires the gradient and Hessian of the value function at each time step. Towards that goal, $V(\textbf{X}(t_k),t_k)$ in \eqref{v17} will be expanded as
\begin{equation}
\label{vexp}
V(t_k)\approx V(t_k)|_{\bar{\textbf x}}+\delta \textbf{X}(t_k)^\top V_\textbf x(t_k)|_{\bar{\textbf x}}+\frac{1}{2}\delta \textbf{X}(t_k)^\top V_{\textbf{xx}}(t_k)|_{\bar{\textbf x}}\delta \textbf X(t_k).
\end{equation}
By plugging eqs. \eqref{Q}, \eqref{Qexp}, \eqref{d13} and \eqref{vexp} in \eqref{v17} and matching the zero order, first order and second order terms of the resulting expression, we can get the following equations
\begin{equation}
\begin{split}
\label{ric}
&V(t_{k})=Q_0^k+(-\frac{1}{2}\gamma^2+\gamma)(Q_{u}^k)^\top \ell(t_{k}),\\
&V_{\textbf{x}}(t_{k})=Q_{\textbf{x}}^k+\Sigma(t_k)^\top Q_{uu}^kl(t_k)+\Sigma(t_k)^\top Q_{u}^k+Q_{\textbf{x}u}(t_k)\ell(t_{k}),\\
&V_{\textbf{xx}}(t_{k})=Q_{\textbf{xx}}^k+\Sigma(t_k)^\top Q_{uu}^k\Sigma(t_k)+\Sigma(t_k)^\top Q_{ux}^k+Q_{\textbf{x}u}(t_k)\Sigma(t_{k}).
\end{split}
\end{equation}
The equations above are solved backwards in time. The boundary conditions are given by the final cost term. Specifically, $V(\textbf X(t_f))=\bm F(\textbf{X}(t_f))|_{\bar{\textbf x}}$, $V_{\textbf x}(\textbf{X}(t_f))=\bm F_\textbf x(\textbf{X}(t_f))|_{\bar{\textbf x}}$ and $V_{\textbf{xx}}(\textbf{X}(t_f))=\bm{F}_{\textbf{xx}}(\textbf{X}(t_f))|_{\bar{\textbf x}}$. Once \eqref{ric} has been computed for all time steps, $\delta u^*$ may be calculated as in \eqref{d13} and a new control trajectory $u$ can be determined as in \eqref{newu}. Then, we set $\bar{u}\leftarrow u$ and the procedure is repeated until convergence.

\begin{remark}
Since gPC-DDP works locally about nominal trajectories and does not rely on state space grids, scalability is attained. Hence, although gPC representations increase the dimension of the obtained state vector $\textbf{X}$, a tractable solution to \eqref{gpcstochoptprob3} can be determined.
\end{remark}

Observe that gPC-DDP requires the Jacobian and Hessian of the gPC-based dynamics $\textbf{f}$. The following proposition addresses the computation of $\nabla_\textbf{x}\textbf{f}$ and $\nabla_{\textbf{x}\textbf{x}}[\textbf{f}]_l$ ($l=1,...,\bm n$). The remaining terms can be determined similarly.

\begin{prop}
\label{gpcpropgrad}
	Consider the gPC-based dynamics $\emph{\textbf{f}}$ given in \eqref{gpcX}, and let $\dot{x}_{ij}$ be its $l^{\text{th}}$ element, defined in \eqref{ccc}. Then $\nabla_\bold{x}\bold{f}$ and $\nabla_{\emph{\textbf{x}\textbf{x}}}[\emph{\textbf{f}}]_l$ are given respectively by
	\begin{equation}
	\label{jacquad}
	\begin{split}
	&\nabla_\bold{x}\bold{f}=\\
	&\bigg[{\displaystyle\int_\mathcal{D}} \bigg(\nabla_xf(\xi)\otimes\big(\Phi(\xi)\otimes\Phi(\xi)^\top\big)\bigg)\rho(\xi)\rd\xi\bigg]\odot\bigg(\bold{1}_{n\times n}\otimes\psi\bigg),
	\end{split}
	\end{equation}
	
	\begin{equation}
	\label{hessf}
	\nabla_{\emph{\textbf{x}\textbf{x}}}[\emph{\textbf{f}}]_l={\displaystyle\int_\mathcal{D}} \bigg(\frac{\nabla_{xx}[f]_i(\xi)}{\langle\phi_j,\phi_j\rangle}\phi_j(\xi)\otimes\big(\Phi(\xi)\otimes\Phi(\xi)^\top\big)\bigg)\rho(\xi)\rd\xi,
	\end{equation}
	where $\otimes$ is the Kronecker product, $\odot$ is the Hadamard product and $\bold{1}$ is the unit matrix. Furthermore, $\Phi\in\mathbb{R}^{K+1}$ includes the $K+1$ orthogonal polynomials of the state vector (i.e., $\Phi=(\phi_0(\xi),...,\phi_K(\xi))^\top$), while $\psi=(1/\langle\phi_0,\phi_0\rangle,...,1/\langle\phi_K,\phi_K\rangle)^\top\in\mathbb{R}^{K+1}$.
\end{prop}
\begin{proof}
		Consider the expression given in \eqref{ccc}. Since the gPC coefficients are deterministic, one has
		\begin{equation*}
		\begin{split}
		&\frac{\partial \textbf f_l}{\partial x_{gh}}=\frac{{\displaystyle\int_\mathcal{D}}\frac{\partial f_i(\xi)}{\partial x_{gh}}\phi_j(\xi)\rho(\xi)\rd\xi}{\langle\phi_j,\phi_j\rangle}=\frac{{\displaystyle\int_\mathcal{D}}\frac{\partial f_i(\xi)}{\partial x_g}\frac{\partial x_g}{\partial x_{gh}}\phi_j(\xi)\rho(\xi)\rd\xi}{\langle\phi_j,\phi_j\rangle}\\
		&=\frac{{\displaystyle\int_\mathcal{D}}\frac{\partial f_i(\xi)}{\partial x_g}\phi_h(\xi)\phi_j(\xi)\rho(\xi)\rd\xi}{\langle\phi_j,\phi_j\rangle},
		\end{split}
		\end{equation*}
		where the second equality is due to chain rule and the last one due to \eqref{gpcx}. Similarly, the following can be shown
		
		\begin{equation*}
		\frac{\partial^2 \textbf f_l}{\partial x_{gh}\partial x_{sd}}=\frac{{\displaystyle\int_\mathcal{D}}\frac{\partial^2 f_i(\xi)}{\partial x_g\partial x_s}\phi_d(\xi)\phi_h(\xi)\phi_j(\xi)\rho(\xi)\rd\xi}{\langle\phi_j,\phi_j\rangle},
		\end{equation*}
		where $i,g,s=1,...,n$ and $j,h, d=0,...,K$. Based on these expressions, the result is obtained by writing $\nabla_\bold{x}\bold{f}$ and $\nabla_{\textbf{x}\textbf{x}}[\textbf{f}]_l$ in matrix form.
\end{proof}
When we deal with mechanical systems, we can incorporate the Variational Integrator (VI) of section IV. Assuming that the Lagrangian $L$ and non-conservatives forces $F$ of the original system are known, the propagation phase of gPC-DDP can be implemented through Algorithm \ref{algdel}. It remains to develop a linearization scheme for the DEL equations \eqref{delg1}, \eqref{delg2} to replace eq. \eqref{deltax}. In this direction, one can use the results in \cite{lin} where the structured linearization of discrete systems was obtained. For our case, this is accomplished by differentiating eq. \eqref{delg1} implicitly with respect to $\bold{Q}^{k+1}$, and eq. \eqref{delg2} explicitly with respect to $\bold{\hat{P}}^{k+1}$. Due to space limitations we will only provide the final results and refer the interested reader to \cite{lin}.

To keep notation simple, let $\hat{L}^k$ and $\bold{\hat{F}}^{k\pm}$ denote $\hat{L}_d(\bold{Q}^k,\bold{Q}^{k+1})$ and $\bold{\hat{F}}^\pm_d(\bold{Q}^k,\bold{Q}^{k+1},u^k)$ respectively. Then, one can obtain for the gPC coefficients in \eqref{delg1}, \eqref{delg2} the following
\begin{equation}
\label{dellin}
\begin{split}
\begin{bmatrix}\delta \bold{Q}^{k+1}\\\delta \bold{\hat P}^{k+1}\end{bmatrix}&=\Theta^k_{DEL}\begin{bmatrix}\delta \bold{Q}^k\\\delta \bold{\hat P}^k\end{bmatrix}+B^k_{DEL}\delta u^{k}
+\\
&\frac{1}{2}\begin{bmatrix}\delta \bold{Q}^k\\\delta \bold{\hat P}^k\end{bmatrix}^\top\bigg(\Gamma^k_{DEL}\times_3\begin{bmatrix}\delta \bold{Q}^k\\\delta \bold{\hat P}^k\end{bmatrix}+\Delta^k_{DEL}\times_3\delta u^k\bigg)+\\
&\frac{1}{2}(\delta u^k)^\top\bigg(\Xi^k_{DEL}\times_3\begin{bmatrix}\delta \bold{Q}^k\\\delta \bold{\hat P}^k\end{bmatrix}+\Lambda^k_{DEL}\times_3\delta u^k\bigg),
\end{split}
\end{equation}
where $\Theta^k_{DEL}=\begin{pmatrix}\frac{\partial \bold{Q}^{k+1}}{\partial \bold{Q}^k} & \frac{\partial \bold{Q}^{k+1}}{\partial \bold{\hat P}^k}\\\frac{\partial \bold{\hat P}^{k+1}}{\partial \bold{Q}^k} & \frac{\partial \bold{\hat P}^{k+1}}{\partial \bold{\hat P}^k}\end{pmatrix}$, $B^k_{DEL}=\begin{pmatrix}\frac{\partial \bold{Q}^{k+1}}{\partial u^k}\\\frac{\partial \bold{\hat P}^{k+1}}{\partial u^k}\end{pmatrix}$,\newline$\Gamma^k_{DEL}=\frac{\partial\Theta^k_{DEL}}{\partial[(\bold{Q}^k)^\top,(\bold{\hat P}^k)^\top]}$, $\Xi^k_{DEL}=\frac{\partial B^k_{DEL}}{\partial[( \bold{Q}^k)^\top,(\bold{\hat P}^k)^\top]}$, $\Delta^k_{DEL}=\frac{\partial \Theta^k_{DEL}}{\partial u^k}$ and $\Lambda^k_{DEL}=\frac{\partial B^k_{DEL}}{\partial u^k}$. The first-order terms are given by
\begin{equation*}
\begin{array}{rcl}
\displaystyle{\frac{\partial \bold{Q}^{k+1}}{\partial \bold{Q}^k}}&=&-(\hat{M}^k)^{-1}[D_1D_1\hat{L}^k+D_1\bold{\hat{F}}^{k-}],\\
\displaystyle{\frac{\partial \bold{Q}^{k+1}}{\partial \bold{\hat P}^k}}&=&-(\hat{M}^k)^{-1},\\
\displaystyle{\frac{\partial \bold{Q}^{k+1}}{\partial u^k}}&=&-(\hat{M}^k)^{-1}D_3\bold{\hat{F}}^{k-},\\
\displaystyle{\frac{\partial \bold{\hat P}^{k+1}}{\partial \bold{Q}^k}}&=&[D_2D_2\hat{L}^k+D_2\bold{\hat{F}}^{k+}]\displaystyle{\frac{\partial \bold{Q}^{k+1}}{\partial \bold{Q}^k}}+D_1D_2\hat{L}^k+D_1\bold{\hat{F}}^{k+},\\
\displaystyle{\frac{\partial \bold{\hat P}^{k+1}}{\partial \bold{\hat P}^k}}&=&[D_2D_2\hat{L}^k+D_2\bold{\hat{F}}^{k+}]\displaystyle{\frac{\partial \bold{Q}^{k+1}}{\partial \bold{\hat P}^k}},\\
\displaystyle{\frac{\partial \bold{\hat P}^{k+1}}{\partial u^k}}&=&[D_2D_2\hat{L}^k+D_2\bold{\hat{F}}^{k+}]\displaystyle{\frac{\partial \bold{Q}^{k+1}}{\partial u^k}}+D_3\bold{\hat{F}}^{k+},\\
\hat{M}^k&=&D_2D_1\hat{L}^k+D_2\bold{\hat{F}}^{k-}.
\end{array}
\end{equation*}
The second-order terms included in $\Gamma^k_{DEL}$, $\Delta^k_{DEL}$, $\Lambda^k_{DEL}$ and $\Xi^k_{DEL}$, are defined in Appendix \ref{appdel}. Also, the inversion of matrix $\hat{M}^k$ is guaranteed by using the implicit function theorem on \eqref{delg1} \cite{lin}.

When the VI is incorporated, we only have to modify the $Q$-functions in \eqref{Qddp} by using \eqref{dellin}. Specifically, we will substitute $\Theta\leftarrow\Theta_{DEL}$, $B\leftarrow B_{DEL}$, $\Delta t\nabla_{\textbf{x}u}\textbf f\leftarrow\Xi_{DEL}$, $\Delta t\nabla_{\textbf{xx}}\textbf f\leftarrow\Gamma_{DEL}$, $\Delta t\nabla_{u\textbf{x}}\textbf f\leftarrow\Delta_{DEL}$ and $\Delta t\nabla_{uu}\textbf f\leftarrow\Lambda_{DEL}$.

Finally, note that computing the linearization scheme in \eqref{dellin}, requires defining the Jacobians and Hessians of $\hat{L}^k$, $\bold{\hat{F}}^{k\pm}$ with respect to the gPC coordinates. Below we give an expression for $D_1D_1\hat{L}^k$. The remaining terms can be determined similarly.

\begin{prop}
	The Hessian $D_1D_1\hat{L}^k$ is given by
	\begin{equation}
	\label{kaioken}
	\begin{split}
	&D_1D_1\hat{L}^k=\\
	&{\displaystyle\int_\mathcal{D}}\bigg(D_1D_1L_d(q^k,q^{k+1},\xi)\otimes\big(\Phi(\xi)\otimes\Phi(\xi)^\top\big)\bigg)\rho(\xi)\rd(\xi),
	\end{split}
	\end{equation}
	where $\Phi$ has been defined in Proposition \ref{gpcpropgrad}.
\end{prop}
\begin{proof}
	First, observe from \eqref{lddd} and \eqref{Lhat} that $\hat{L}^k=\int_\mathcal{D}L_d(q^k,q^{k+1},\xi)\rho(\xi)\rd\xi$. Then, by following an approach similar to the proof in Proposition \ref{gpcpropgrad} the entries of $D_1D_1\hat{L}^k$ can be computed as
	
	\begin{equation*}
	\frac{\partial^2 \hat{L}^k}{\partial q_{gh}^k\partial q_{ij}^k}=\int_\mathcal{D}\frac{\partial^2 L_d(q^k,q^{k+1},\xi)}{\partial q_i^k\partial q_g^k}\phi_h(\xi)\phi_j(\xi)\rho(\xi)\rd\xi,
	\end{equation*}
	where $i,g=1,...,N$ and $j,h=0,...,K$. Then the result can be derived by writing $D_1D_1\hat{L}^k$ in matrix form.
\end{proof}

We conclude this section by summarizing the required steps of gPC-DDP in Algorithm \ref{gpc0}. 
\begin{algorithm}[h]
	\SetAlgoLined
	\KwData{Dynamics and cost functions of \eqref{gpcstochoptprob3}, nominal control trajectory $\bar{u}$;}
	Get the nominal state trajectory $\overline{\textbf X}$ under controls $\bar{u}$ by propagating a discrete version of \eqref{gpcX} forward in time (or by using Algorithm 1, if the VI is employed);\\
	\Repeat{\emph{convergence}}{
	Linearize dynamics about $\bar{u}$, $\overline{\textbf X}$ using eq. \eqref{deltax} (or eq. \eqref{dellin}, if the VI is employed);\\
	Backpropagate $V$, $V_x$, $V_{xx}$ using \eqref{ric};\\
	Pick a specific value for the line search parameter $\gamma$, with $0<\gamma\leq1$. Starting from $k=0$ with $\delta \textbf X^0=0$, compute a new control $u^k$ as in \eqref{newu}. Get the next state $\textbf X^{k+1}$ as described in step 1. Define $\delta\textbf X^{k+1}$ and repeat the procedure for the entire time horizon. Given the obtained state trajectory $\textbf X$ and control sequence $u$, compute the total cost $\bm J$. Accept the new controls if cost reduction is attained. Otherwise, reduce $\gamma$ and repeat;\\
        Update $\bar{u}\leftarrow u$, $\overline{\textbf X}\leftarrow\textbf X$;
}
	\caption{gPC-DDP}
	\label{gpc0}
\end{algorithm}

\subsection{Convergence analysis of gPC-DDP}\label{secconvg}
The authors in \cite{liao000} showed that under some mild assumptions, the original Differential Dynamic Programming algorithm will converge to a stationary solution, regardless of the initial state. Moreover, \cite{ddp0} proves that DDP achieves locally quadratic convergence rates. Unfortunately, the analysis in the latter paper deals only with scalar dynamic systems and does not consider a terminal cost in the problem formulation. Here, we extend the aforementioned works and provide the convergence properties of gPC-DDP for the generic optimal control problem in \eqref{gpcstochoptprob3}.

The following set of assumptions is necessary for our analysis.
\begin{ass}
\label{assgpc}
i) The dynamics and cost functions of \eqref{gpcstochoptprob3} are differentiable up to the third order over a compact convex set $\mathcal{R}\in\mathbb{R}^{\bm n+m}$, ii) $(\emph{\textbf{X}}(t_k),u(t_k))\in\mathcal{R}$, $\forall t_k\in[t_0,t_f]$, iii) $Q_{uu}$ in \eqref{Qddp} remains positive definite for all time instants and iterations of gPC-DDP.
\end{ass}

The convergence properties of gPC-DDP are established by the next theorems.
\begin{theorem}
Consider the optimal control problem in \eqref{gpcstochoptprob3} and suppose Assumption \ref{assgpc} holds. Then, the gPC-DDP algorithm will globally converge to a stationary solution for arbitrary initialization. Moreover, when the line search parameter, $\gamma$, is small enough, the total cost reduction at each iteration is given by
\begin{equation*}
\bm J^{(l)}-\bm J^{(l-1)}=(-\gamma+\frac{\gamma^2}{2})\sum_{k=0}^{K_f-1}(Q_u^k)^\top (Q_{uu}^k)^{-1}Q_u^k,
\end{equation*}
with $\bm J^{(l)}$ being the total cost at iteration ``$l$", and $K_f$ being the discretized time horizon.
\end{theorem}
\begin{proof}
 Since the cost function and dynamic constraints in \eqref{gpcstochoptprob3} are deterministic, the proof is similar to \cite{liao000} and is omitted due to page restriction.
\end{proof}
\begin{theorem}\label{thconv}
Suppose Assumption \ref{assgpc} holds. Let also $U^*=((u^{*0})^\top,(u^{*1})^\top,...)^\top$ denote the stationary point gPC-DDP converges to and $U^{(l)}$ denote the controls obtained at iteration ``$l$". Then, $\exists$ $c>0$ and $l$ large enough, such that the following holds
\begin{equation}
\label{convrate}
||U^{(l)}-U^*||\leq c||U^{(l-1)}-U^*||^2.
\end{equation}
Therefore, the convergence rate will be locally quadratic.
\end{theorem}
\begin{proof}
The proof is in Appendix \ref{appconv}.
\end{proof}

\begin{remark}
The main assumption required by the analysis is that $Q_{uu}^k$ remains positive definite $\forall k$. When this is not ensured a priori, one can use a shift parameter $\theta>0$ and set \cite{Nocedal}
\begin{equation*}
Q_{uu}^k\leftarrow Q_{uu}^k+\theta I,
\end{equation*}
where $\theta$ is selected large enough so that the required condition is enforced. An adaptive method of picking $\theta$ was also proposed in \cite{ddp0}.
\end{remark}

\subsection{Practical implementation of gPC-DDP} \label{gggb}
In the propagation and linearization phases of gPC-DDP, one has to compute integrals of the form $\int_{\mathcal{D}}\mathsf f(\xi)\rho(\xi)\rd\xi$, where $\mathsf f\in L_\rho^2$ (e.g., in \eqref{ccc}, \eqref{jacquad}, \eqref{hessf} or \eqref{kaioken}). These quantities can rarely be calculated analytically; except if linear or sufficiently simple dynamics are considered. In light of this, numerical methods have to be utilized. Note that expressions of this type can be viewed as expectations. Hence, one could sample points from $\rho(\xi)$ and compute a Monte Carlo approximation. Unfortunately, a great number of function evaluations is usually required to get a good estimate. To reduce computational complexity, one can instead use Gaussian quadrature (also called Gaussian cubature when $\xi$ is multi-dimensional). Based on the corresponding density function $\rho(\xi)$, properly selected nodes $z_{l_i}$ and weights $w_{l_i}$ can be determined such that
\begin{equation}
\label{gaussquad}
\int_{\mathcal{D}_1}^...\int_{\mathcal{D}_d}\mathsf f(\xi)\rho(\xi)\rd\xi\approx\sum_{l_1=1}^{l_{gq}}...\sum_{l_d=1}^{l_{gq}}w_{l_1}...w_{l_d}\mathsf f(z_{l_1},...,z_{l_d}),
\end{equation}
where $\xi=(\xi_1,...,\xi_d)^\top$ and $\xi_i,z_i\in\mathbb{R}$. The sets of nodes and weights can be obtained by solving an eigenvalue decomposition problem \cite{quad}. One important aspect is that the above formula holds with equality when the integrand is a polynomial of degree less than or equal to $2l_{gq}-1$. Since we consider smooth dynamics in this paper, the particular integration method is expected to be highly accurate. Lastly, note that this full-tensor method requires $l_{gq}^d$ function evaluations. When the set of random inputs is large, one can use sparse quadrature formulas, as presented in \cite{smol}, to reduce the number of nodes $z_{l_i}$ while retaining sufficient precision.

It is therefore deduced that implementing the propagation and linearization phases of gPC-DDP can be readily done by evaluating expressions of the nominal dynamics at the different quadrature nodes. Before moving to the simulations section, we summarize the most important attributes of gPC-DDP that are not offered by conventional stochastic trajectory-optimization methods.
\begin{itemize}
\item We derive our framework by relying on the Differential Dynamic Programming method. An iterative, scalable and fast-converging scheme is developed that is capable of controlling the density of the states. Specifically, desired trajectories for the expected states can be defined, while also affecting higher order moments. Moreover, no particular assumptions have to be made about the structure of the system dynamics. We only require differentiability of $f$ with respect to $x$ and $u$ in order to perform linearizations.
\item gPC-DDP employs Polynomial Chaos theory to represent stochasticity. Consequently, multiple types of uncertainty in the parameters and initial states can be handled, as the ones presented in Table \ref{tab:ask}. In case some random inputs follow arbitrary probability distributions, one can build proper orthogonal polynomials by using the Stieltjes procedure \cite{gpca} and incorporate them in our framework.
\item When mechanical systems are considered, Variational Integrators can be designed to obtain faithful discrete representations. This can further increase the numerical accuracy of our methodology and, as shown in the simulations section, reduce its computational complexity by allowing coarse discretizations of the time horizon.
\end{itemize}

\section{Simulated examples}\label{secsim}
First, a low dimensional system is considered, namely the Duffing oscillator. This will allow us to illustrate the algorithm in detail and make some simplifications when applying the gPC expansions. Subsequently, we apply gPC-DDP on a quadrotor to highlight the broad applicability of the framework.

\subsection{Example: Duffing oscillator}
The dynamics of the Duffing oscillator are given by
\begin{equation}
\label{duff}
\dot{x}=\begin{pmatrix} \dot{x}_1\\\dot{x}_2\end{pmatrix}=\begin{pmatrix}x_2\\-\lambda x_1-\frac{1}{4}x_2-x_1^3+u\end{pmatrix}.
\end{equation}
The uncertainty of this system lies in the parameter $\lambda$, as well as the initial state $x_1(t_0)$. We consider these quantities to be normally distributed with $\lambda\sim\mathcal{N}(\mu_\lambda,\sigma_\lambda^2)$ and $x_1(t_0)\sim\mathcal{N}(\mu_1^0,(\sigma_1^0)^2)$. Based on Table \ref{tab:ask}, Hermite polynomials will be employed. The first few (unnormalized) Hermite polynomials are given by \cite{bookxiu}
\begin{equation}
\label{herm}
\phi_0(\mathsf x)=1,\quad\phi_1(\mathsf x)=\mathsf x,\quad\phi_2(\mathsf x)=\mathsf x^2-1,...,\quad\mathsf x\in\mathbb{R},
\end{equation}
Furthermore, some important expressions are
\begin{equation*}
\rho(\mathsf x)=\frac{1}{\sqrt{2\pi}}\exp^{-\mathsf x^2/2},\quad\langle\phi_i,\phi_j\rangle=\delta_{ij}i!,
\end{equation*}
\begin{equation*}
\langle\phi_i,\phi_j,\phi_g\rangle = \begin{cases}
\frac{g!j!i!}{(s-g)!(s-j)!(s-i)!}, &i,j,g\text{ even $\&$ $s\geq\max(i,j,g)$}\\
0, &\text{otherwise}
\end{cases}
\end{equation*}
where $s=\frac{i+j+g}{2}$.

By applying \eqref{gpcl}, \eqref{gpcx0}, $\lambda$ and $x_1(t_0)$ can be fully described as
\begin{equation}
\label{lam}
\lambda(\xi^p)=\mu_\lambda\phi_0^p(\xi^p)+\sigma_\lambda\phi_1^p(\xi^p),
\end{equation}
\begin{equation}
\label{x10gpc}
x_1^0(\xi^0)=\mu_1^0\phi_0^0(\xi^0)+\sigma_1^0\phi_1^0(\xi^0),
\end{equation}
where $\xi^p$, $\xi^0$ are standard Normal variables (i.e., $\xi^p,\xi^0\sim\mathcal{N}(0,1)$). In this setting, the states are  influenced by $\xi=(\xi^p,\xi^0)\in\mathbb{R}^2$. Applying the gPC expansion on the state vector yields
\begin{equation}
\label{gpcd}
x_1(t,\xi)\approx\sum_{j=0}^K x_{1j}(t)\phi_j(\xi),\hspace{4mm}x_2(t,\xi)\approx\sum_{j=0}^K x_{2j}(t)\phi_j(\xi),
\end{equation}
with
\begin{equation}
\label{gpcr}
\phi_j(\xi)=\phi^p_{j_1}(\xi^p)\phi_{j_2}^0(\xi^0), \quad 0\leq j_1,j_2\leq r,
\end{equation}
where $r$ is the maximum order of each univariate orthogonal polynomial. Since $x_2(t_0)$ is deterministic, we take $x_{20}(t_0)=x_2(t_0)$ and $x_{2j}(t_0)=0$ for $j>0$. Plugging \eqref{lam}, \eqref{gpcd} in \eqref{duff} and performing Galerkin projection as in \eqref{ccc} gives
\begin{equation}
\label{gpcduff}
\begin{split}
&\dot{x}_{1k}=x_{2k},\\
&\dot{x}_{2k}=\frac{1}{\langle\phi_k,\phi_k\rangle}\bigg(-\sum_{i=0}^Kx_{1i}\big(\mu_\lambda\langle\phi_0^p,\phi_k,\phi_i\rangle+\sigma_\lambda\langle\phi_1^p,\phi_k,\phi_i\rangle\big)-\\
&\frac{1}{4}x_{2k}\langle\phi_k,\phi_k\rangle-\sum_{i=0}^K\sum_{g=0}^K\sum_{j=0}^Kx_{1i}x_{1g}x_{1j}\langle\phi_k,\phi_i,\phi_g,\phi_j\rangle+\langle\phi_k\rangle u\bigg),
\end{split}
\end{equation}
for $k=0,...,K$. Note also that one can analytically compute the derivatives of the above expression, which will be required by the linearization phase of gPC-DDP.

Subsequently, the gPC-DDP algorithm is applied. We consider the case where $\mu_\lambda=3$ and $\sigma_\lambda=0.1$. For the initial states we have $\mu_1^0=4$, $\sigma_1^0=0.08$ and $x_2^0=0$. The task is to reach the target state $x^{goal}=(3,0)^\top$ in $t_f=1.8$ sec. We picked quadratic cost functions as in \eqref{gpcrunnL}, \eqref{gpctermF}, with weighting matrices for the terminal cost defined as in \eqref{gpc_Sf0}, \eqref{gpcsf0001}. Regarding the latter, the weights are set to $s_{f_{10}}=s_{f_{20}}=400$ (associated with expected states) and $s_{f_{1j}}=300$, $s_{f_{2j}}=100$, for $j>0$ (associated with variance). The running cost was set to $\bm L=\frac{1}{2}0.01u^2$. Hence, gPC-DDP will penalize trajectories with i) expected terminal states far from the desired ones, ii) high terminal variance. Fig. \ref{nnn4} shows how $||U^{(l)}-U^*||$ changes at each iteration ``$l$", with $U^*$ being the solution that gPC-DDP converges to. We initialized the algorithm by selecting zero controls (i.e., $\bar{u}(t_k)=0$, for $k=0,1,...$).

\begin{figure}[t]
	\centering
	\includegraphics[width=0.4\textwidth]{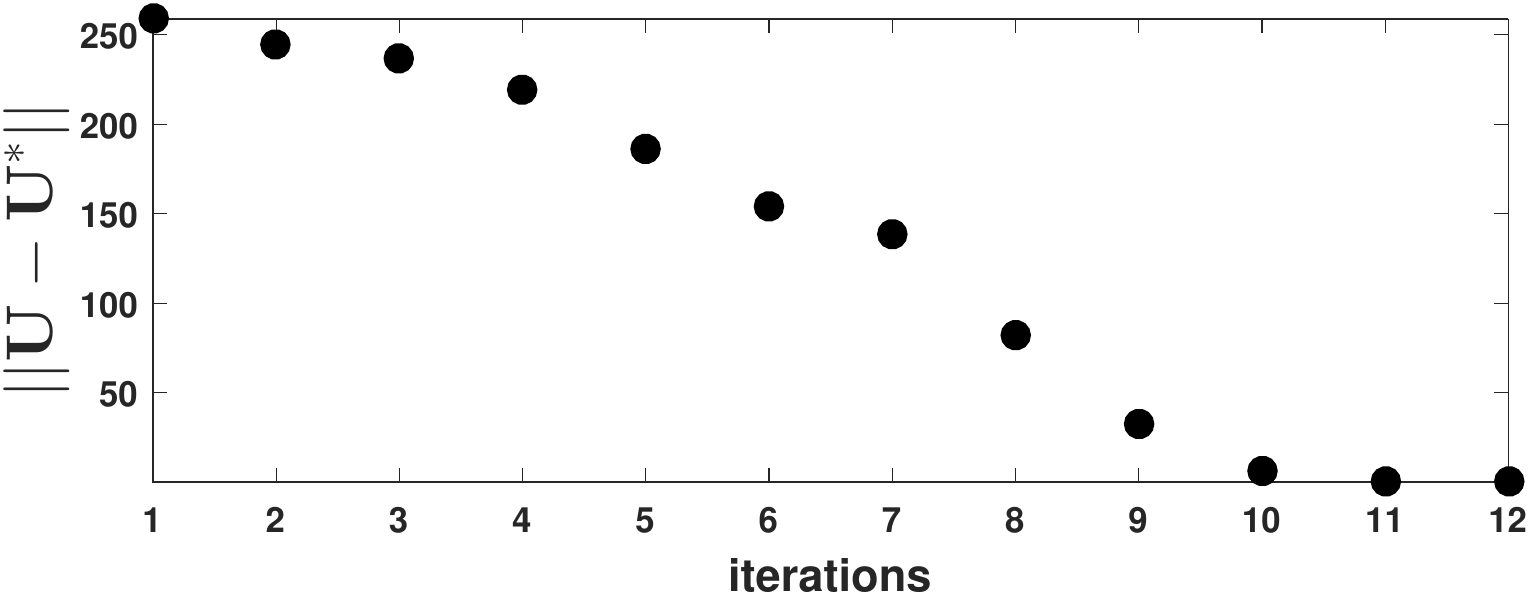}
	\caption{Duffing oscillator: Difference between the controls obtained by gPC-DDP at each iteration and the locally optimal solution the algorithm converges to.}
	\label{nnn4}
\end{figure}

\begin{figure}
	\centering
	\subfigure[The Monte Carlo and gPC estimates of the expected states are illustrated, along with $\pm3\sigma$ of sampled trajectories (green and red shaded areas). gPC-DDP is capable of guiding the expected states to the target, while also minimizing the terminal variance.]{
		\includegraphics[width=0.48\textwidth]{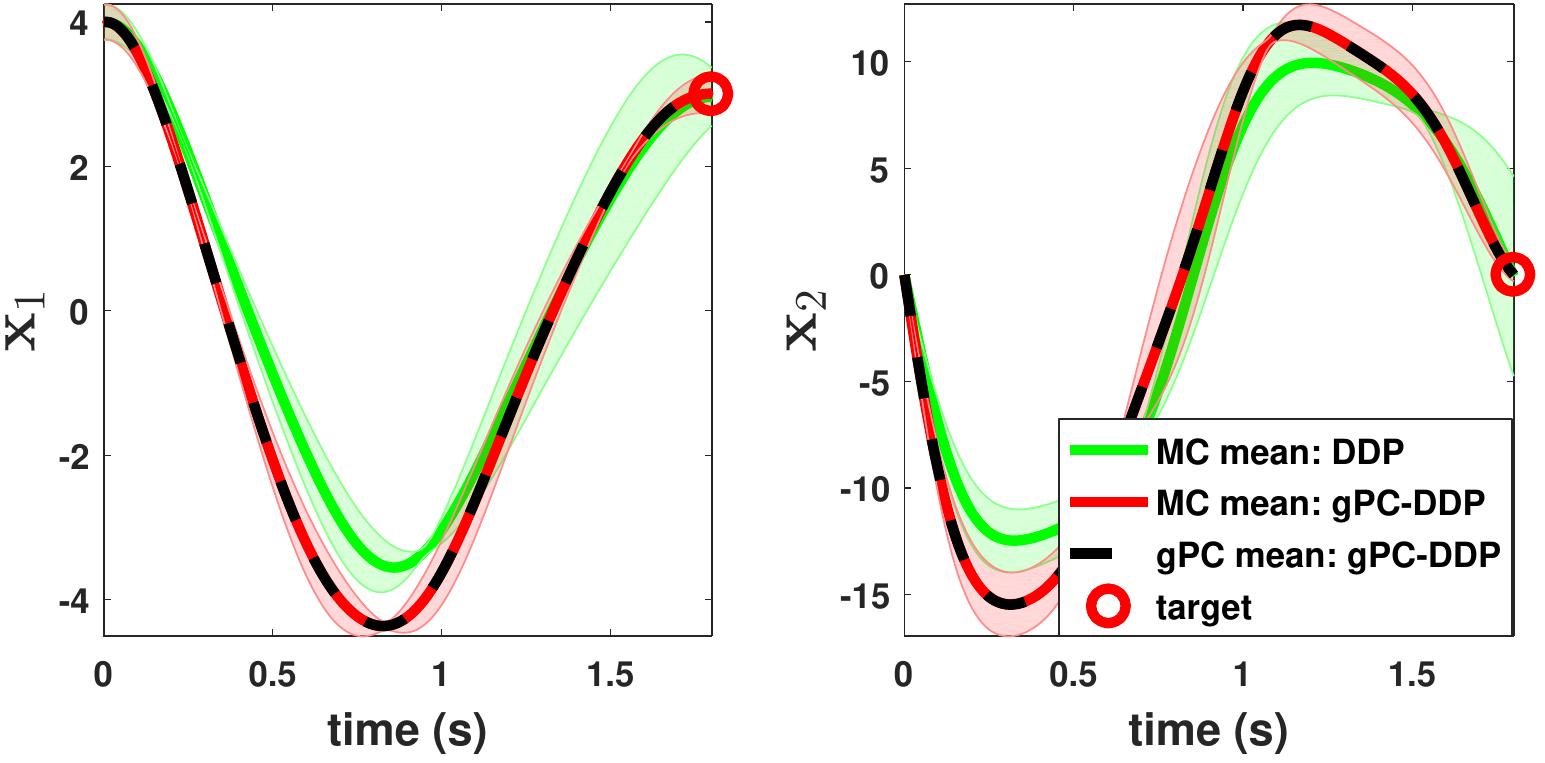}\label{nnn0}}
	\subfigure[Variance of the (sub)optimal state trajectories obtained by gPC-DDP and deterministic DDP respectively. The gPC and Monte Carlo estimates are provided.]{
		\includegraphics[width=0.4741\textwidth]{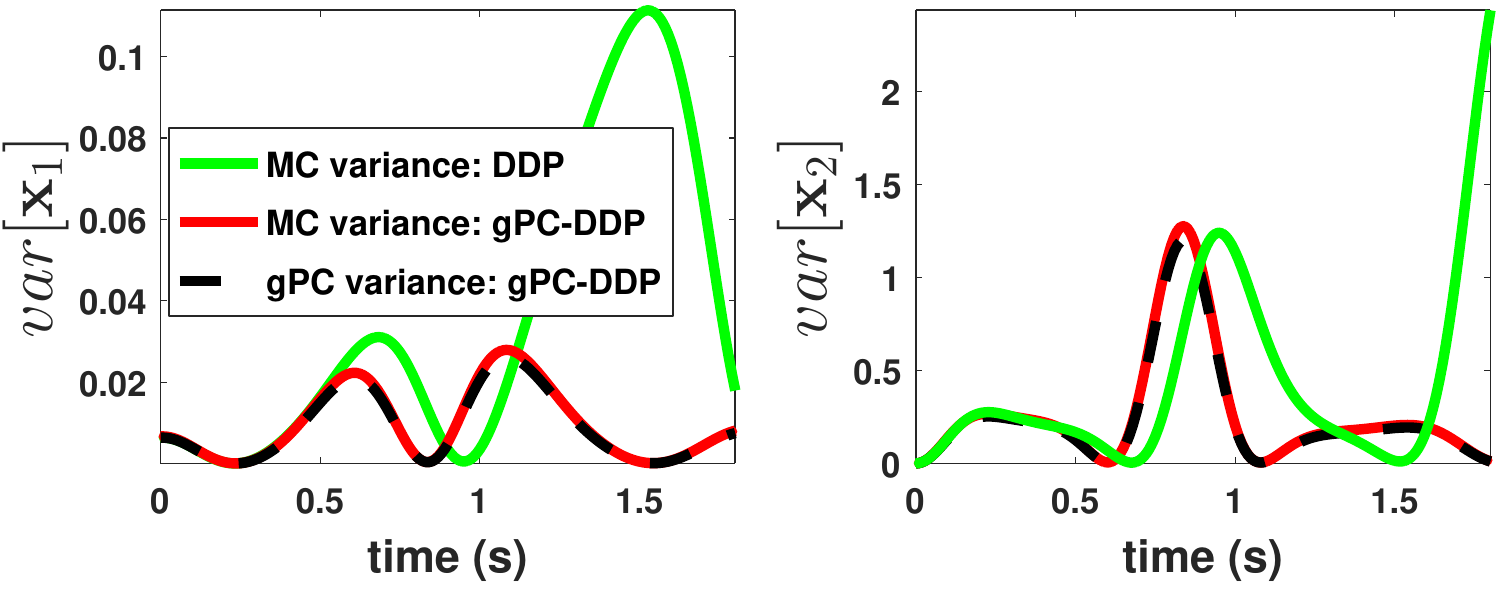}\label{nnn1}}
	\caption{Duffing oscillator: Comparison between gPC-DDP and deterministic DDP (applied for the mean parameter values).}
\end{figure}

We compare our algorithm with the original Differential Dynamic Programming method applied for the mean values of the uncertain parameters and cost functions defined as in \eqref{runnL}, \eqref{termF}. Fig. \ref{nnn0} shows the evolution of the gPC and Monte Carlo estimates of the mean, as well as $\pm 3\sigma$ of 1000 trajectories under the (locally) optimal controls. Furthermore, Fig. \ref{nnn1} depicts the gPC and Monte Carlo variance estimates of the states under the different control settings. The results clearly highlight the versatility of gPC-DDP in terms of finding trajectories that not only guide the system to the target but also reduce state uncertainty. For the gPC expansions, we selected $r=3$ terms in \eqref{gpcr} resulting in an $\bm n=20$ state vector. Due to space limitations, we do not show how the accuracy of the gPC series approximation changes with the number of expansion terms $K$. In general, the estimated moments get closer to the actual ones as $K$ increases, with the convergence rate depending on the quantity to be approximated (see discussion in section \ref{secgpc}, as well as \cite{xiu}, \cite{fisher}, \cite{bookxiu}).

For the implementation of gPC-DDP we utilized a standard Euler scheme to discretize the dynamics in \eqref{gpcduff} and performed the corresponding linearization in \eqref{deltax}. Next, we incorporate the Variational Integrator of section \ref{secvigpc} and call the obtained scheme ``VI-gPC-DDP".

The Lagrangian and non-conservative forces for the Duffing oscillator are given respectively by
\begin{equation*}
L(q,\dot{q})=\frac{1}{2}\dot{q}^2-\frac{\lambda}{2}q^2-\frac{1}{4}q^4, \quad F(q,\dot{q},u)=u-\frac{1}{4}\dot{q}.
\end{equation*}
The gPC expansions of the generalized coordinates are
\begin{equation*}
q\approx\sum_{j=0}^K q_j\phi_j(\xi),\hspace{3mm}\dot{q}\approx\sum_{j=0}^K \dot{q}_j\phi_j(\xi).
\end{equation*}
One can determine through \eqref{Lhat} and \eqref{gpcf} the Lagrangian function and non-conservative forces for the gPC coordinates. Specifically
\begin{equation*}
\begin{split}
&\hat{L}=-\frac{1}{2}\sum_{j=0}^K\sum_{g=0}^Kq_jq_g\big(\mu_\lambda\langle\phi_0^p,\phi_j,\phi_g\rangle+\sigma_\lambda\langle\phi_1^p,\phi_j,\phi_g\rangle\big)+\\
&\frac{1}{2}\sum_{j=0}^K\dot{q}_j^2\langle\phi_j,\phi_j\rangle-\frac{1}{4}\sum_{l=0}^K\sum_{h=0}^K\sum_{g=0}^K\sum_{j=0}^Kq_lq_hq_gq_j\langle\phi_l,\phi_h,\phi_g,\phi_h\rangle,
\end{split}
\end{equation*}
\begin{equation*}
\hat{F}_j=u\langle\phi_j\rangle-\frac{1}{4}\dot{q}_j\langle\phi_j,\phi_j\rangle,\quad j=0,...,K.
\end{equation*}
The discrete versions of the above expressions, $\hat{L}^d$ and $\hat{F}_{dij}^\pm$ can be determined via \eqref{lddd} and \eqref{fddd} respectively. Furthermore, the linearizations required by \eqref{dellin}, can be performed analytically.

To show the importance of a variational integration scheme, we compare the (locally) optimal trajectories given by gPC-DDP $\&$ VI-gPC-DDP for varying time steps $\Delta t$. In general, a properly discretized model will behave similarly to its continuous counterpart, even for relatively large values of $\Delta t$. In that case, we would expect that the time-step selection would induce small deviations in the solution of the optimal control problem. Fig. \ref{nnn3} shows the gPC mean estimates after implementing the algorithms on the Duffing oscillator. It can be easily deduced that utilizing Variational Integrators highly reduces the impact of $\Delta t$ on the obtained trajectories.

\begin{figure}
	\centering
	\subfigure[gPC-DDP]{\includegraphics[width=0.440000001\textwidth]{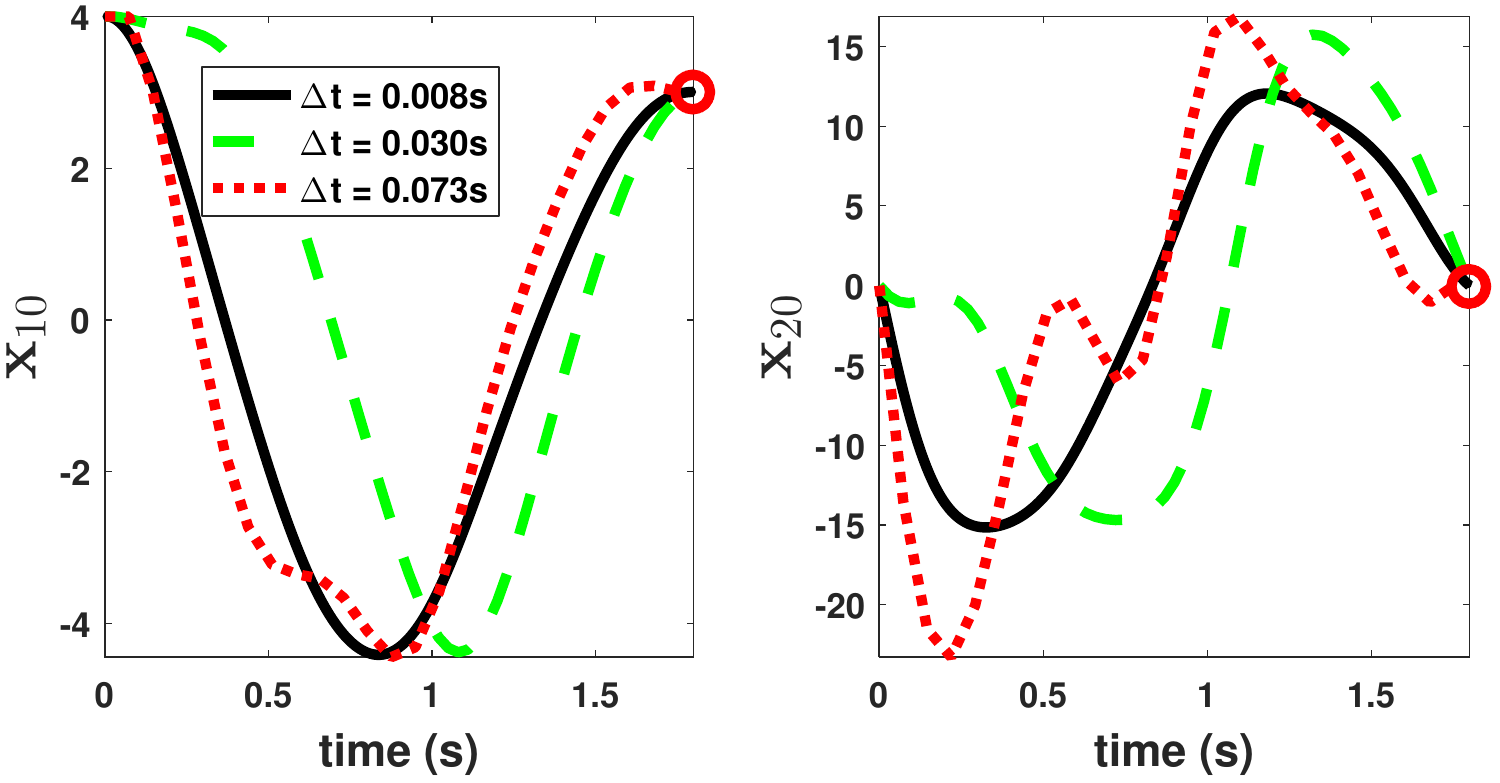}
	}
	\subfigure[VI-gPC-DDP]{
		\includegraphics[width=0.44\textwidth]{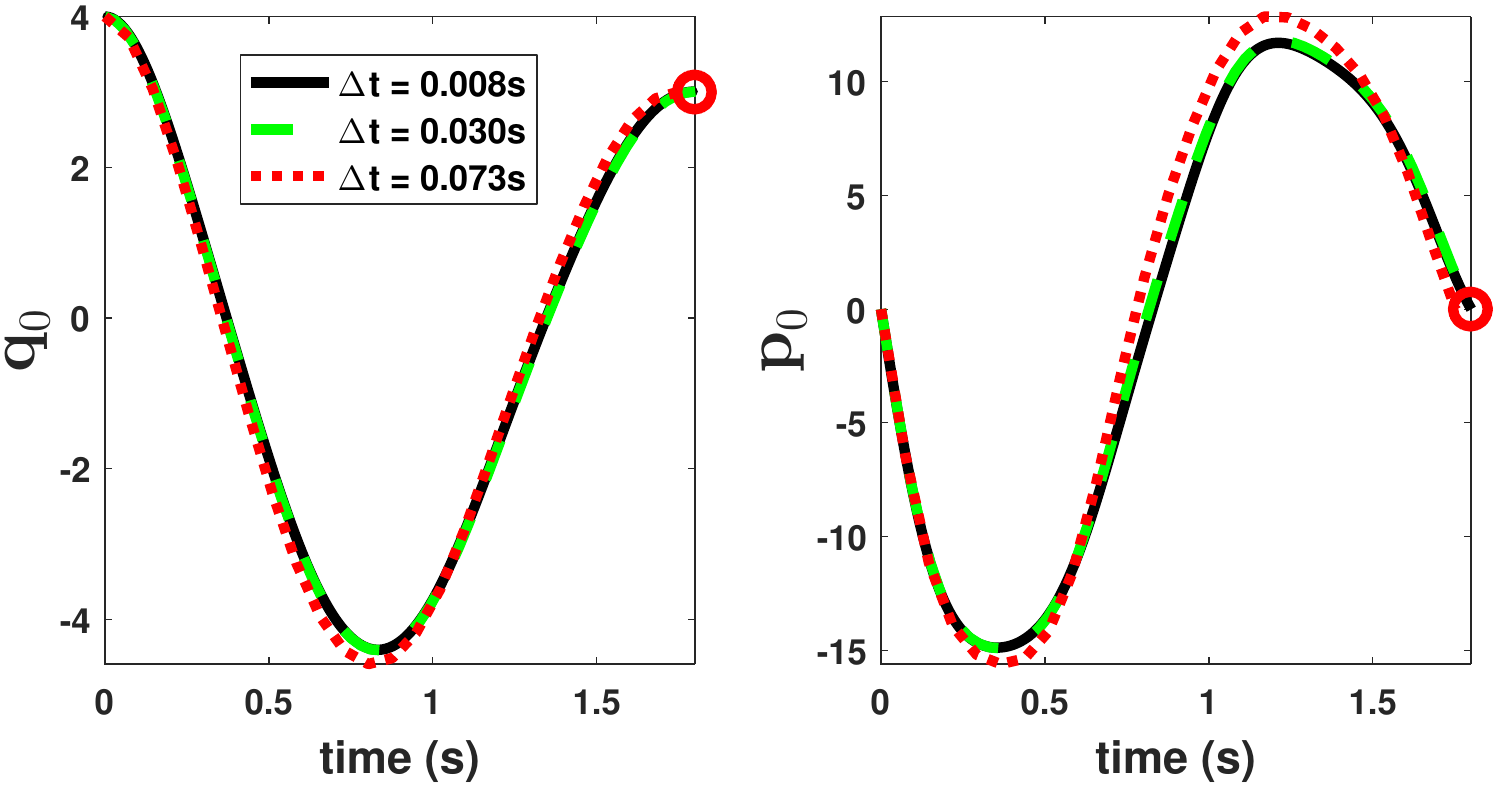}
	}
	\caption{Duffing oscillator: Comparison between gPC-DDP and VI-gPC-DDP for different step sizes $\Delta t$. For the former, an explicit Euler scheme was used to propagate and linearize the gPC-based dynamics. The gPC mean estimates are able to reach the target for all cases. However, VI-gPC-DDP is much more insensitive to the selection of $\Delta t$.}
	\label{nnn3}
\end{figure}

\subsection{Example: quadrotor}
The state vector considered here is $(\chi^\top,\eta^\top,\dot \chi^\top,\dot \eta^\top)^\top\in\mathbb{R}^{12}$, where $\chi=(x,y,z)^\top\in\mathbb{R}^3$ and $\eta=(\phi,\theta,\psi)^\top\in\mathbb{R}^3$ denote the position and Euler angles respectively (see Fig. \ref{quadill}). The Lagrangian and external forces of the system are given respectively by
\begin{figure}[b]
	\centering
	\includegraphics[width=0.19\textwidth]{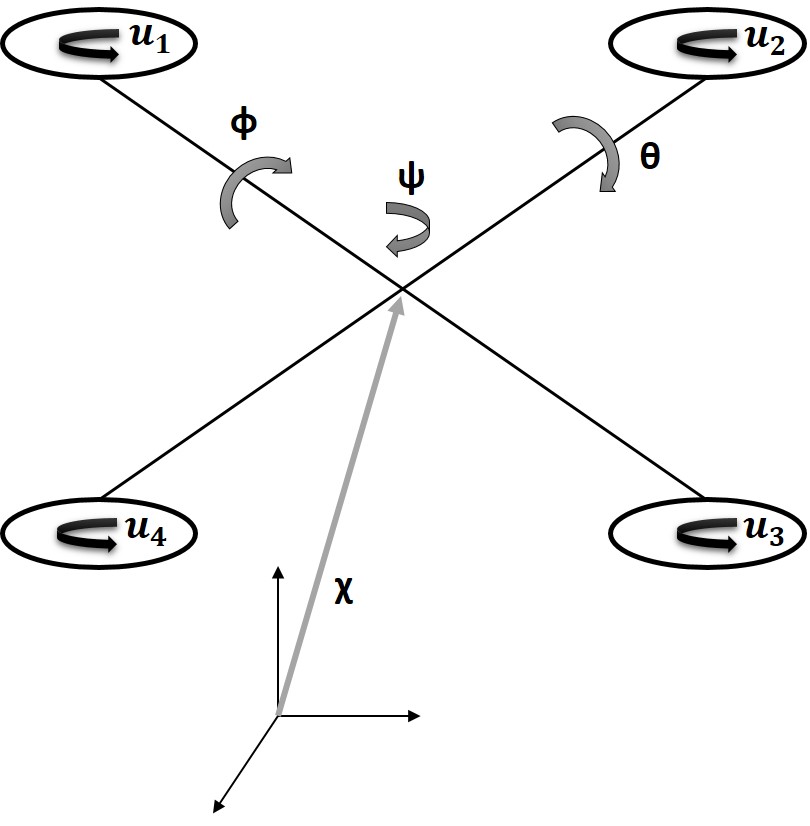}
	\caption{Illustration of the quadrotor system.}
	\label{quadill}
\end{figure}
\begin{equation*}
L=\frac{1}{2}m\dot \chi^\top\dot \chi-mgz+\frac{1}{2}\dot \eta \mathbb{J}(\eta ) \eta, \quad F=\begin{bmatrix}F_\chi\\ \tau \end{bmatrix},
\end{equation*}
\begin{equation*}
F_\chi=R(\eta)\begin{bmatrix}0\\0\\\sum_{i=1}^4u_i\end{bmatrix}, \quad \tau=\begin{bmatrix}(u_3-u_1)lG_{tr}/G_{rot}\\(u_2-u_4)lG_{tr}/G_{rot}\\u_1+u_3-u_2-u_4\end{bmatrix},
\end{equation*}
where $u_i$ are the control inputs, $l$ is the distance between the rotors and center of mass, $m$ is the mass, $g$ is the gravitational acceleration and $G_{tr}$, $G_{rot}$ are constants that depend on the air density and the shape of the rotors. Furthermore, $\mathbb{J}(\eta)=W(\eta)^\top \mathbb{I}W(\eta)$ is the Jacobian matrix from angular velocities to $\dot \eta$, with $\mathbb{I}=\text{diag}(\mathbb I_x,\mathbb I_y,\mathbb I_z)$ being the inertia matrix and $W(\eta)$ being a transformation matrix for angular velocities, from the inertial frame to the body frame. Finally, $R(\eta)$ denotes a rotation matrix from the body frame to the inertial frame. The equations of motion for this system can be found by applying the {\it Euler-Lagrange} expressions in \eqref{lagr}. Specifically, these will be
\begin{equation*}
m\ddot \chi+mg\begin{bmatrix}0\\0\\1\end{bmatrix}=F_\chi,\quad\mathbb{J}\ddot\eta+C(\eta,\dot \eta)\dot \eta=\tau,
\end{equation*}
where $C(\eta,\dot \eta)=\dot {\mathbb{J}}-\frac{1}{2}\frac{\partial (\dot \eta^\top\mathbb{J})}{\partial\eta}$. A detailed derivation of these expressions can be found in \cite{quadmod}.

Fig. \ref{par_quad} shows the parameters that were used for the simulations. We consider the case where the exact values of $G_{tr}$ and $G_{rot}$ are not known. Specifically, we assume that they are uniformly distributed with known upper and lower bounds.
\begin{figure}[b]
	\begin{minipage}{0.45\linewidth}\centering
		\captionof{table}{Deterministic parameters}
		\begin{tabular}{|c||c|}
			\hline
			\scalebox{0.7}{$g$}&\scalebox{0.7}{$9.81$ $m/s^2$}\\
			\scalebox{0.7}{$m$}&\scalebox{0.7}{$1$ $kg$}\\
			\scalebox{0.7}{$l$}&\scalebox{0.7}{$0.24$ $m$}\\
			\scalebox{0.7}{$\mathbb I_x$}&\scalebox{0.7}{$8.1\cdot 10^{-3}$ $kgm^2$}\\
			\scalebox{0.7}{$\mathbb I_y$}&\scalebox{0.7}{$8.1\cdot 10^{-3}$ $kgm^2$}\\
			\scalebox{0.7}{$\mathbb I_z$}&\scalebox{0.7}{$14.2\cdot 10^{-3}$ $kgm^2$}\\
			\hline
		\end{tabular}
	\end{minipage}
	\hspace{0.5cm}
	\begin{minipage}{0.45\linewidth}
		\centering
		\captionof{table}{Uniformly distributed, stochastic parameters}
		\begin{tabular}{|c||c|c|}
			\hline
			&min&max\\
			\hline
			\scalebox{0.7}{$G_{tr}$}&\scalebox{0.7}{$2.85\cdot10^{-5}$}&\scalebox{0.7}{$2.95\cdot 10^{-5}$}\\
			\scalebox{0.7}{$G_{rot}$}&\scalebox{0.7}{$1.05\cdot10^{-6}$}&\scalebox{0.7}{$1.15\cdot 10^{-6}$}\\
			\hline
		\end{tabular}
	\end{minipage}
	\caption{Parameter values for the quadrotor}
	\label{par_quad}
\end{figure}

Since the random parameters follow a Uniform distribution, we choose Legendre polynomials (see Table \ref{tab:ask}). Some important expressions for these (unnormalized) orthogonal polynomials are provided below \cite{bookxiu}
\begin{equation*}
\phi_0(\mathsf x)=1,\quad\phi_1(\mathsf x)=\mathsf x,\quad\phi_2(\mathsf x)=\frac{3}{2}\mathsf x^2-\frac{1}{2},...,\quad\mathsf x\in\mathbb{R},
\end{equation*}
\begin{equation*}
\rho(\mathsf x)=1,\quad\langle\phi_i,\phi_j\rangle=\frac{2}{2i+1}\delta_{ij}.
\end{equation*}
For the Polynomial Chaos expansions of the states we picked $r=2$ terms in \eqref{gpcexpansion1}. This resulted in an $\bm n=72$ state vector. Note that in this problem, the propagation and linearization phases of gPC-DDP cannot be computed analytically. Hence, Gaussian quadrature has to be used, as presented in section \ref{gggb}. In these formulas, we used $l_{gq}=3$ quadrature nodes for each uncertain parameter.

The task we consider for gPC-DDP, is reaching the target state $((\chi^{goal})^\top,(\eta^{goal})^\top,(\dot \chi^{goal})^\top,(\dot \eta^{goal})^\top)=\textbf 0_{12}$ starting from the deterministic state $(-3,3,3,\textbf0_9)$.  We picked quadratic cost functions as in \eqref{gpcrunnL}, \eqref{gpctermF}, with weighting matrices for the terminal cost defined as in \eqref{gpc_Sf0}, \eqref{gpcsf0001}. The weights associated with expected states were set to $s_{f_{i0}}=8$. The running cost was set to $\bm L=\frac{1}{2}0.1\sum_{i=1}^4u_i^2$. Finally, we initialized the algorithms by selecting zero controls (i.e., $\bar{u}_i(t_k)=0$, for $k=0,1,...$).

\begin{figure}
	\centering
	\includegraphics[width=0.489\textwidth]{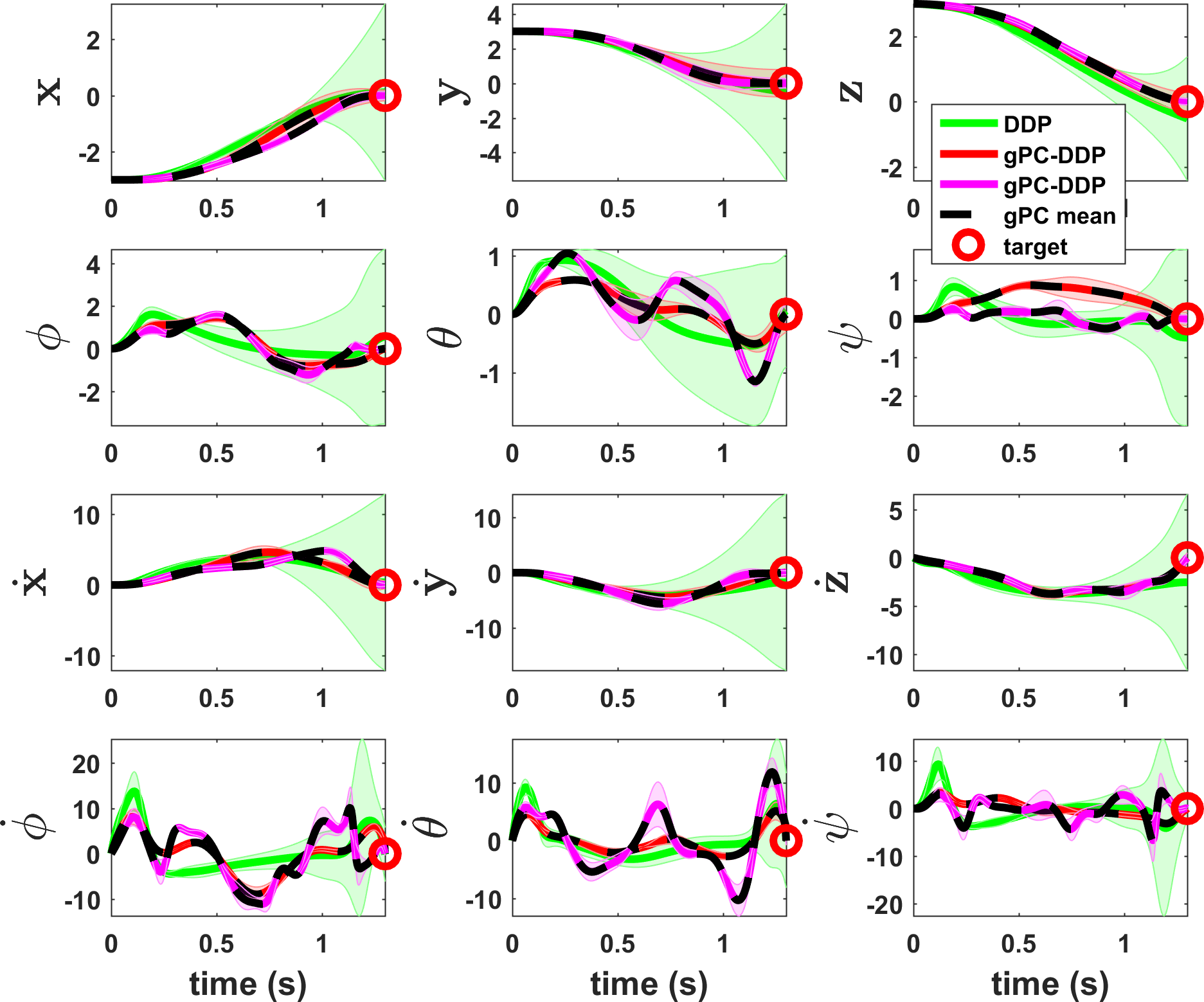}
	\caption{Quadrotor: Comparison between gPC-DDP and deterministic DDP. Regarding the former, two settings are considered with different uncertainty penalization levels (red: low, magenta: high). Solid lines represent Monte Carlo mean estimates, while black dashed lines represent gPC mean estimates. The colored shaded areas correspond to $\pm3\sigma$ of trajectories sampled under the different control sequences (units - $(x,y,z)$: $m$, $(\phi,\theta,\psi)$: rad, $(\dot{x},\dot{y},\dot{z})$: $m/s$, $(\dot \phi,\dot \theta,\dot \psi)$: rad$/s$).}
	\label{quad3}
\end{figure}
\begin{figure}
	\centering
	\subfigure{\includegraphics[width=0.13\textwidth]{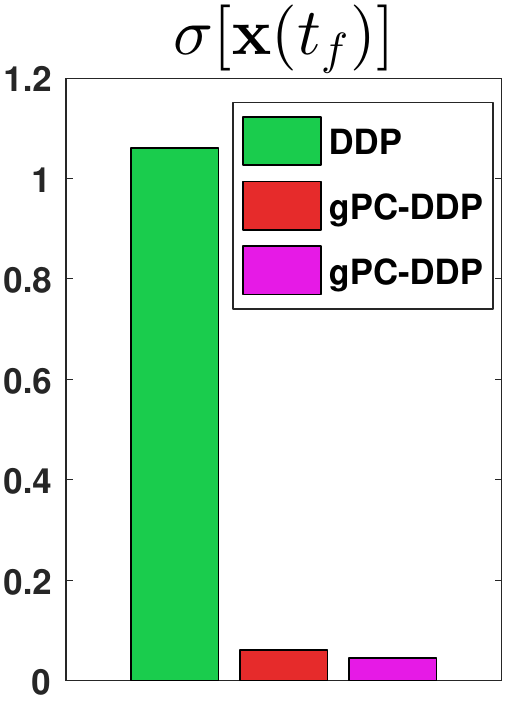}
	}
	\subfigure{
		\includegraphics[width=0.13\textwidth]{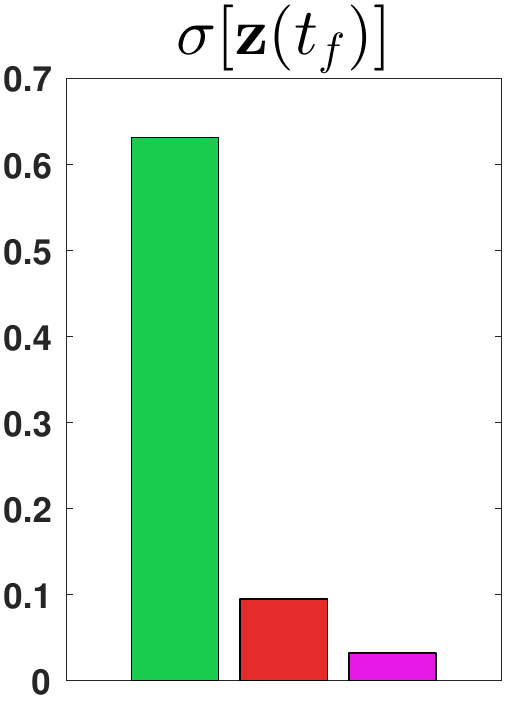}
	}
	\subfigure{
		\includegraphics[width=0.121\textwidth]{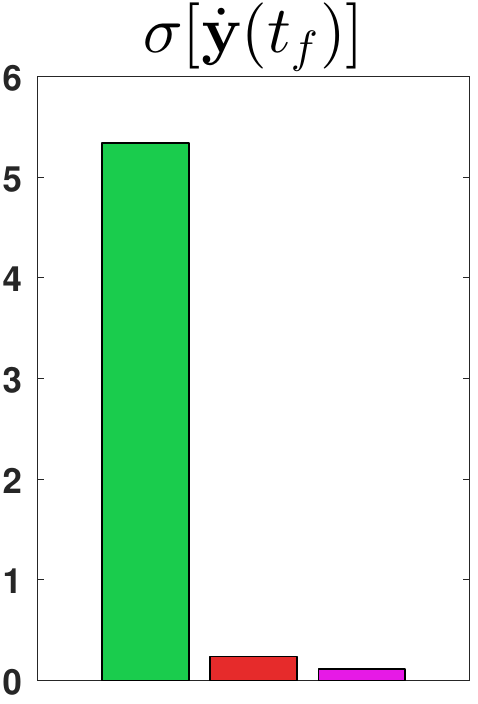}
	}
	\caption{Quadrotor: Standard deviation of $x$, $z$, $\dot y$ at the terminal time, obtained by deterministic DDP and gPC-DDP respectively. Red and magenta bars correspond to different uncertainty penalization levels (red: low, magenta: high).}
	\label{quad4}
\end{figure}
\begin{figure}
	\centering
	\includegraphics[width=0.408\textwidth]{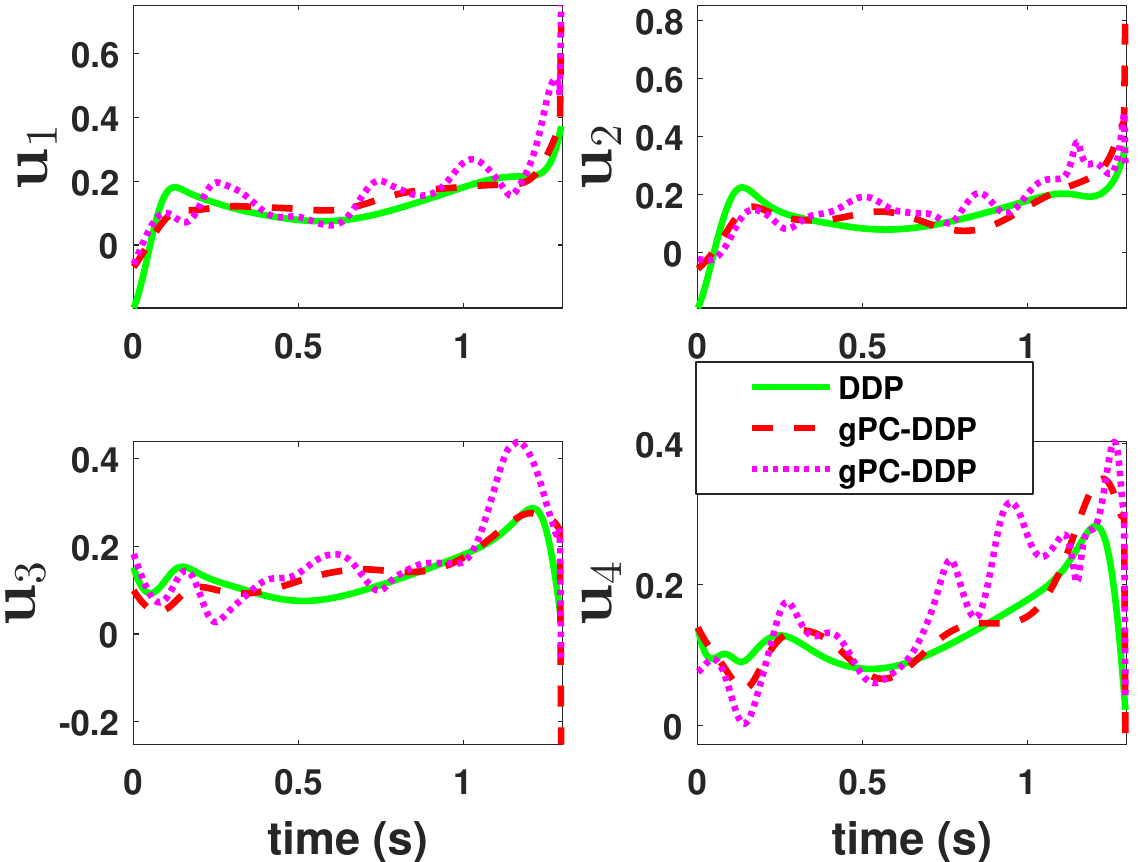}
	\caption{Quadrotor: Controls obtained by deterministic DDP and gPC-DDP for different uncertainty penalization levels (red: low, magenta: high).}
	\label{quad8}
\end{figure}

\begin{figure*}[!t]
	\centering
	\includegraphics[width=0.13\textwidth]{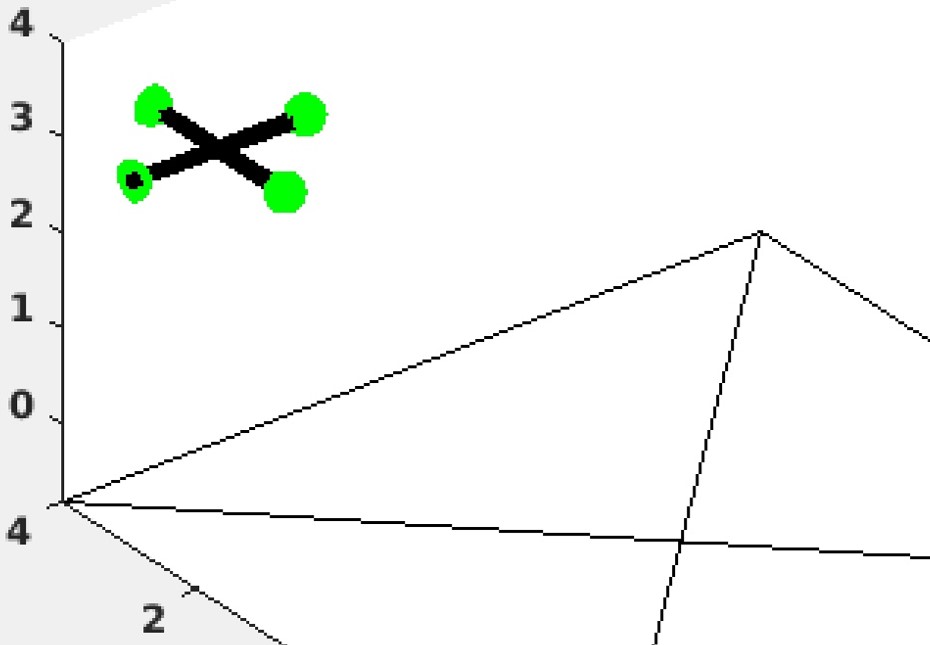}
	\includegraphics[width=0.13\textwidth]{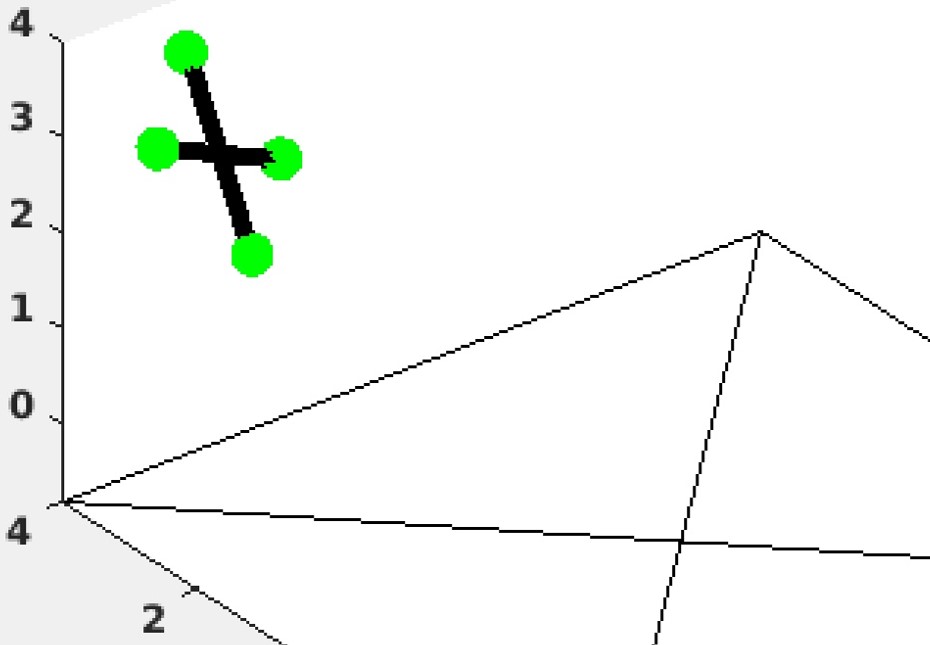}
	\includegraphics[width=0.13\textwidth]{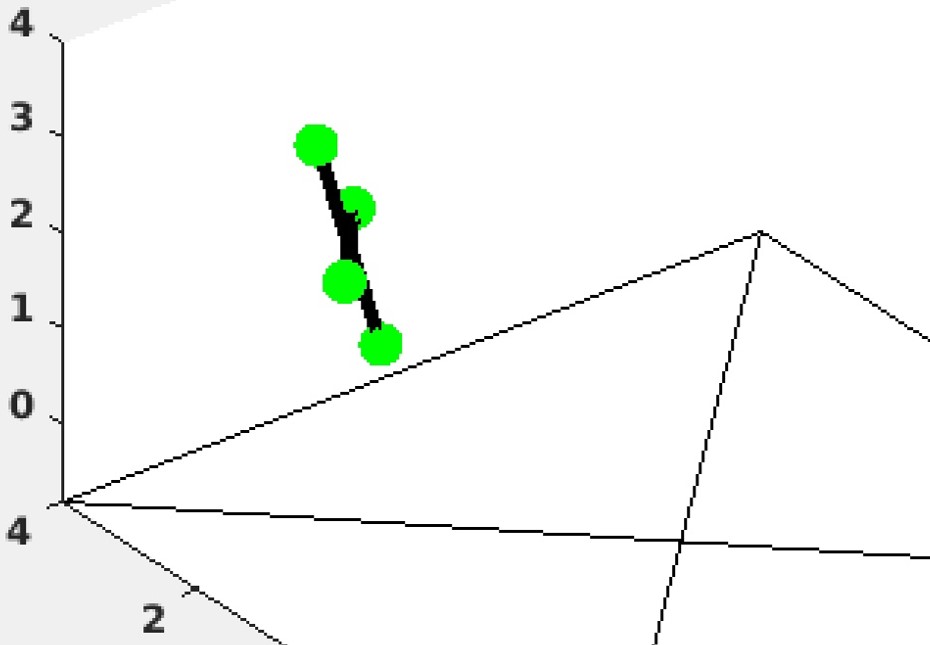}
	\includegraphics[width=0.13\textwidth]{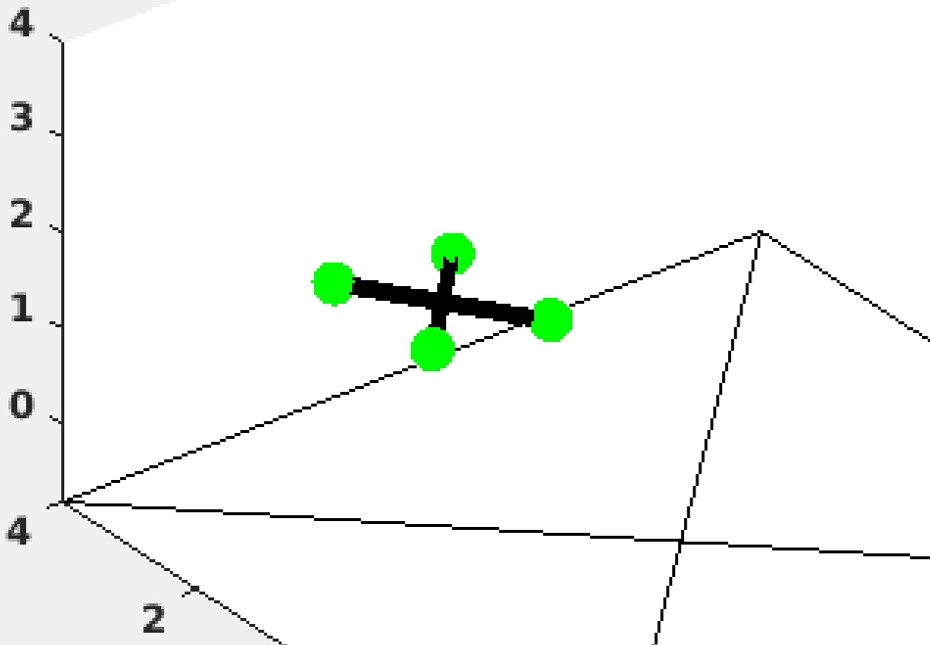}
	\includegraphics[width=0.13\textwidth]{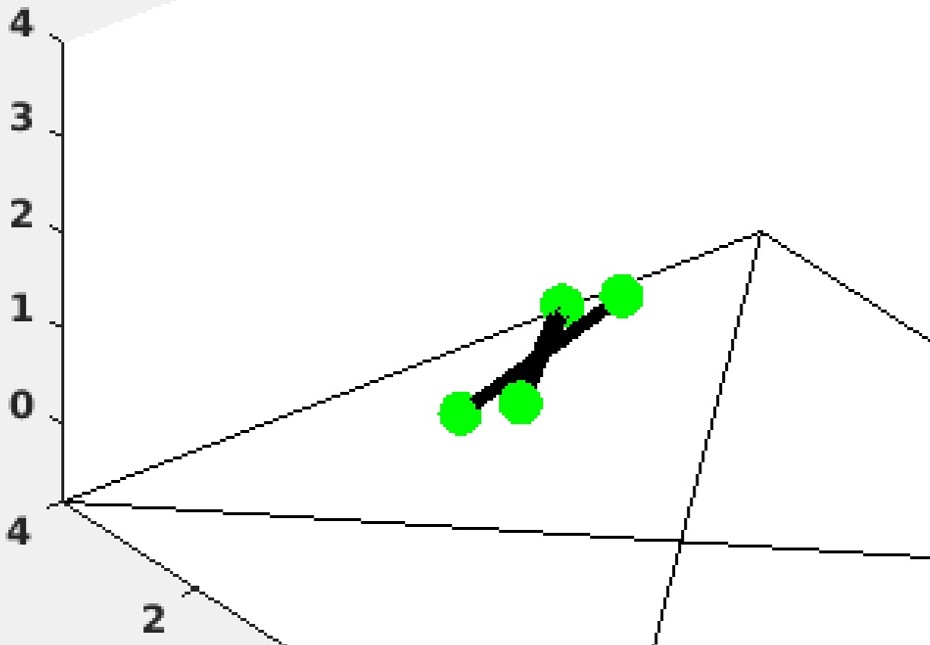}
	\includegraphics[width=0.13\textwidth]{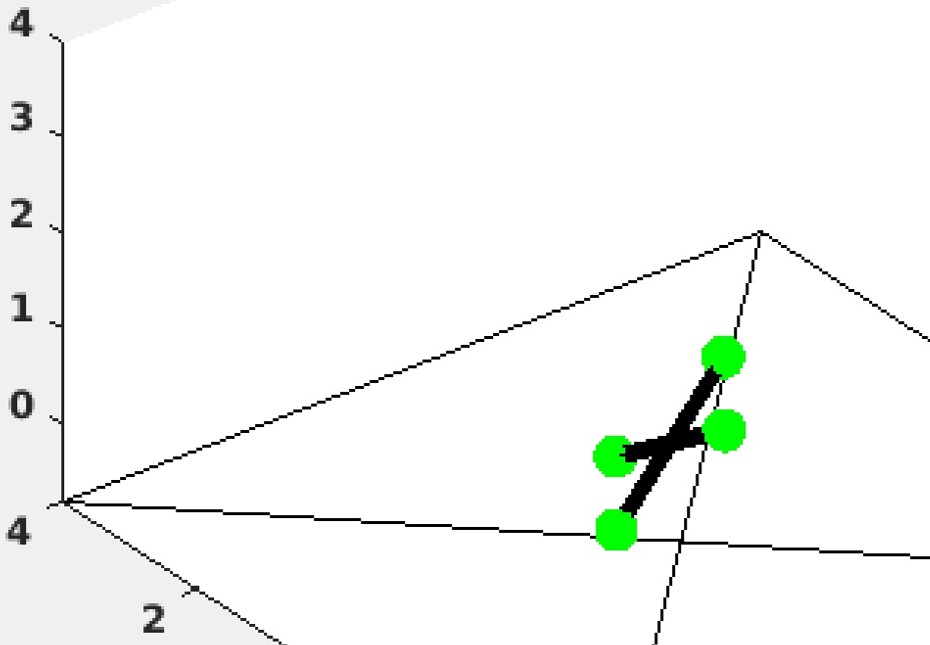}
	\includegraphics[width=0.13\textwidth]{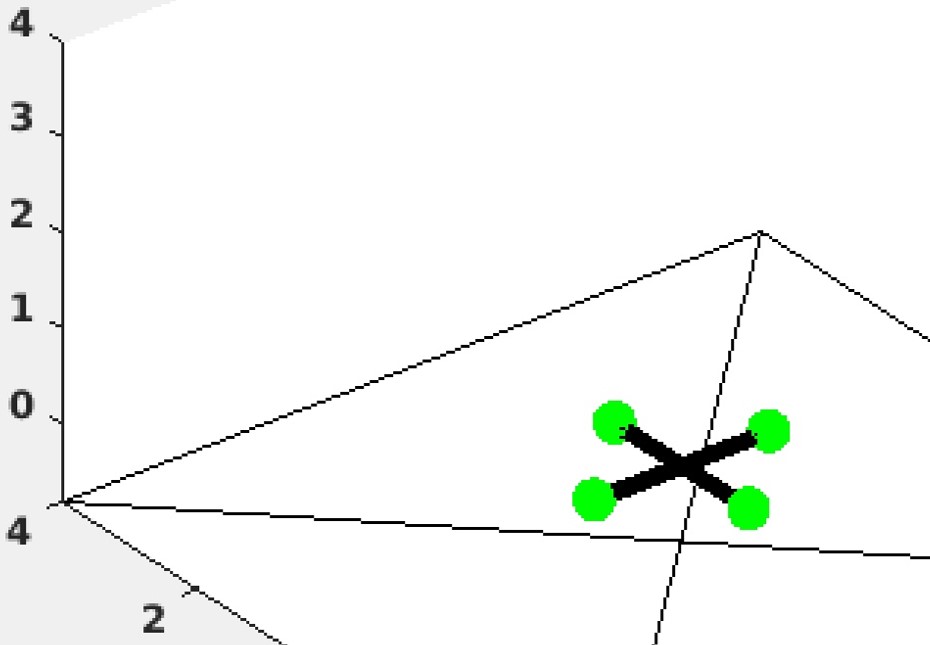}
	
	\caption{Quadrotor: Instances of the (sub)optimal, mean state trajectory obtained by gPC-DDP.}
	\label{quad_ill1}
\end{figure*}

\begin{figure}
	\centering
	\includegraphics[width=0.37\textwidth]{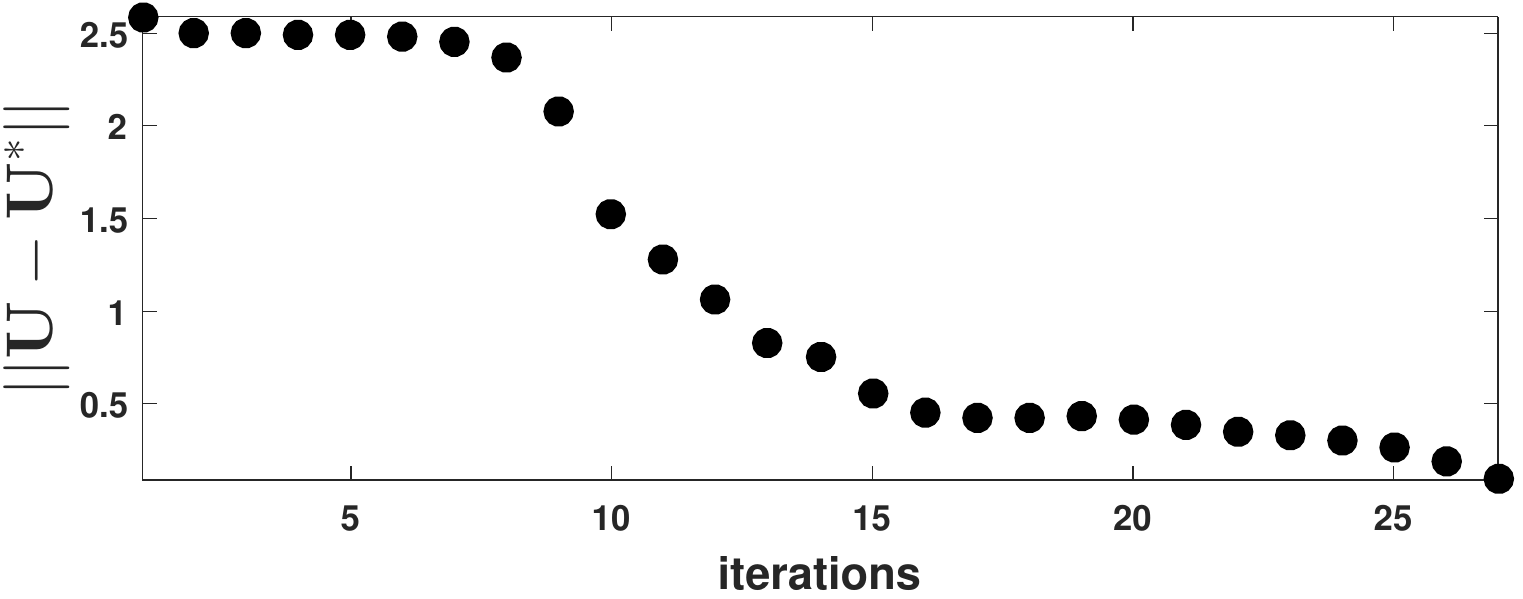}
	\caption{Quadrotor: Difference between the controls obtained by gPC-DDP at each iteration and the locally optimal solution the algorithm converges to.}
	\label{quaddu}
\end{figure}

Fig. \ref{quad3} shows the obtained state trajectories after gPC-DDP has converged. To illustrate the impact of uncertainty penalization, we consider two settings with different (variance-related) weights $s_{f_{ij}}$, $j=1,...,K$. Specifically, trajectories in red correspond to $s_{f_{ij}}=0.0013$, while those in magenta correspond to $s_{f_{ij}}=0.044$. As before, we compare our results to the original DDP algorithm applied for the mean values of $G_{tr}$ and $G_{rot}$. Fig. \ref{quad4} shows the terminal standard deviation of three states (MC estimate), while  Fig. \ref{quad8} displays the obtained control inputs under the different optimization settings. Finally, Fig. \ref{quad_ill1} provides an illustration of the quadrotor task under gPC-DDP's (sub)optimal expected trajectories, while Fig. \ref{quaddu} shows how $||U^{(l)}-U^*||$ changes at each iteration ``$l$" (both of these figures were generated for $s_{f_{ij}}=0.0013$ as the variance-related weights).

The results reveal that penalizing state uncertainty in the problem setup can be rather beneficial in trajectory-optimization tasks. It is also important to make the following observation. It was noticed during simulations that large penalization of state uncertainty generally increases the required bandwidth of the controllers. This particular behavior is depicted in Fig. \ref{quad8}; for large values of the variance-related weights, the obtained controls seem to fluctuate more frequently. Thus, one has to make a compromise between uncertainty minimization and control effort when tuning the cost functions.

Next, we incorporate our variational integration scheme. We follow the same procedure as in the Duffing oscillator example; however, obtaining closed form solutions for the gPC-based Lagrangian and non-conservative forces is impossible in this case. Hence, Gaussian quadrature has to be used to estimate the discrete version of $\hat L$ and $\hat{\textbf F}$, as well as their derivatives.

To test the effect of the VI scheme, we consider the following scenario. We solve the optimization problem by using different time steps in the problem formulation. Then, we implement the (sub)optimal controls on a discretized version of the gPC-based dynamics, with a very small discretization step ($\Delta t=0.001$). In this way, the behavior of the continuous dynamics \eqref{gpcX} can be accurately simulated. Fig. \ref{quad9} shows three of the gPC mean estimates after applying the controls obtained by gPC-DDP $\&$ VI-gPC-DDP respectively. For the former, an explicit Euler scheme was used. It can be deduced that when large time steps are employed, naive discretizations of dynamic equations behave much differently than their continuous counterpart and are inappropriate for optimal control tasks.

\begin{figure}
	\centering
	\subfigure[gPC-DDP]{\includegraphics[width=0.48\textwidth]{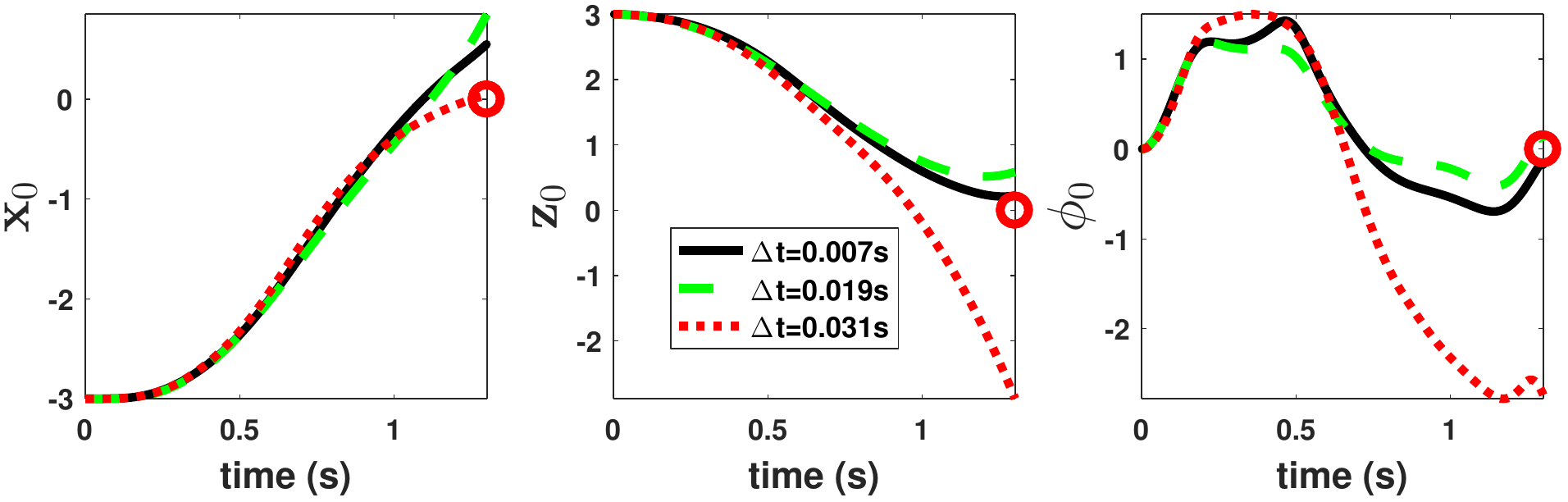}
	}
	\subfigure[VI-gPC-DDP]{\includegraphics[width=0.48\textwidth]{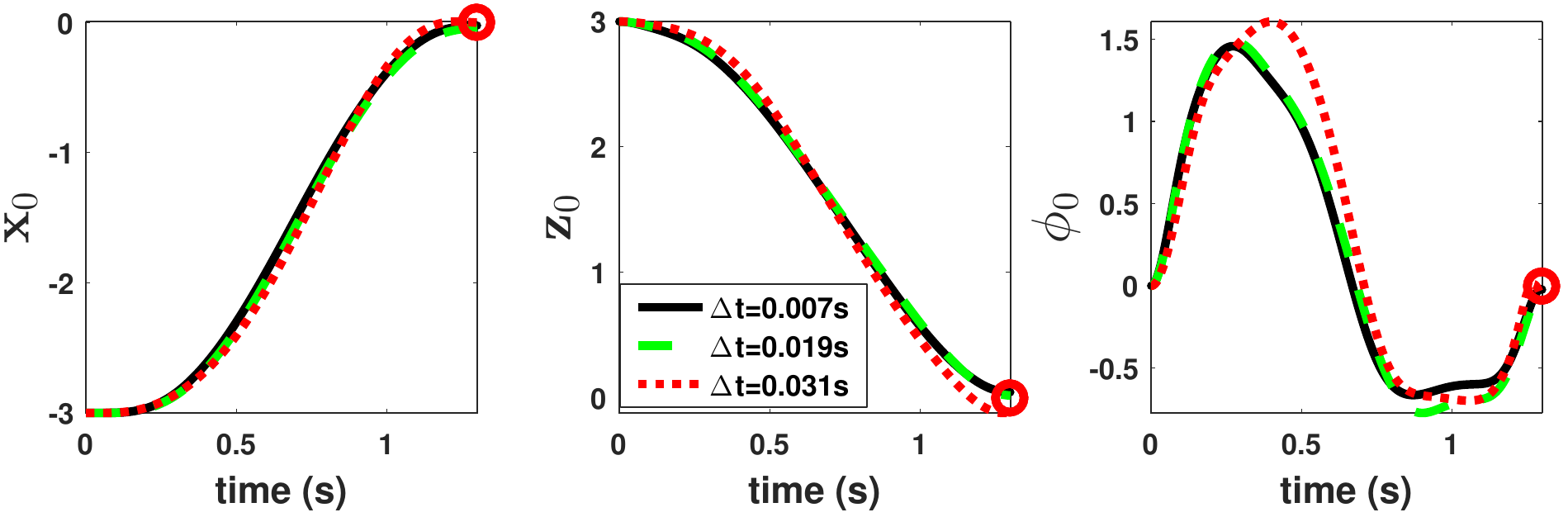}
	}
	\caption{Quadrotor: Comparison between gPC-DDP $\&$ VI-gPC-DDP. The legends show the discretization steps used in the different trajectory-optimization schemes. The plots were obtained by implementing the (locally) optimal controls on a discretized version of the gPC-based dynamics with a very small time step $\Delta t$. In this way, the actual, continuous system can be accurately simulated. Performance is significantly reduced when naive discretizations are employed.}
	\label{quad9}
\end{figure}

In fact, the ability to use a coarse discretization can significantly reduce the computational complexity of our optimal control algorithm. Fig. \ref{quad10} compares VI-gPC-DDP to gPC-DDP in terms of the average time required for the propagation phase. Note that the variational integration scheme will be slower for similar step sizes since one has to solve eq. \eqref{delg2} implicitly, which usually requires more than one iterations. However, when the discrete horizon decreases sufficiently, the overall required time is reduced.

\begin{figure}
	\centering
	\includegraphics[width=0.3\textwidth]{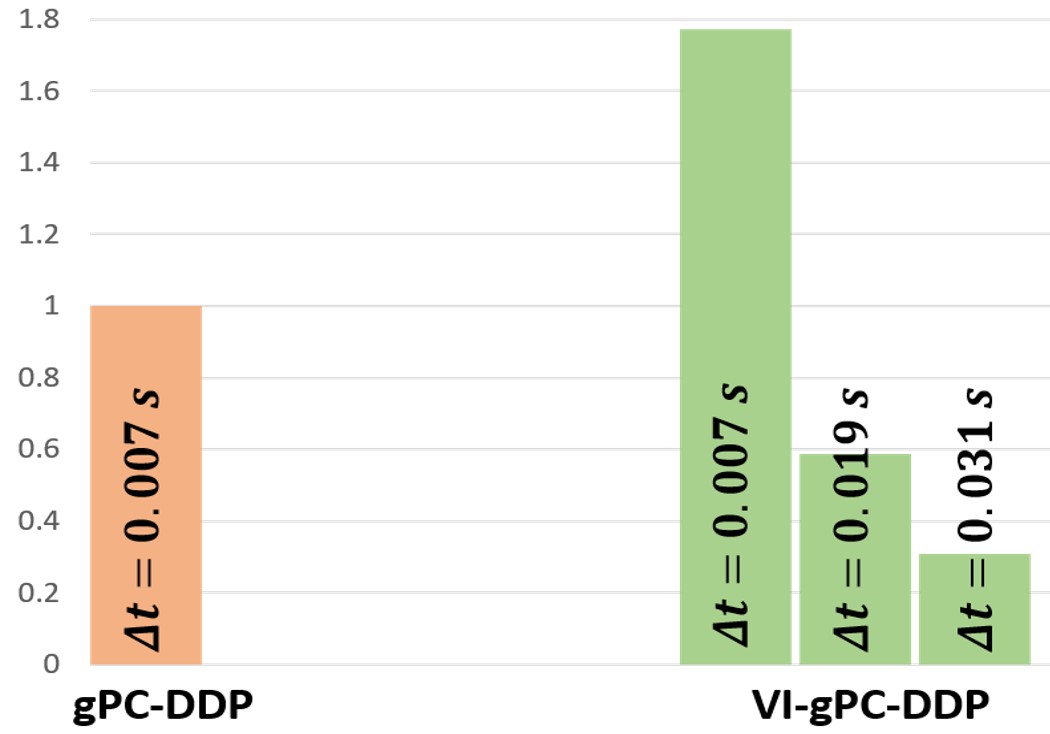}
	\caption{Quadrotor: Average elapsed time for the propagation phase of VI-gPC-DPP, implemented for different values of $\Delta t$. The results are normalized with respect to gPC-DDP's corresponding time, when a small discretization step is used. Utilizing variational integration schemes allows having a coarser time discretization and, therefore, reduces computational complexity.}
	\label{quad10}
\end{figure}

\section{Discussion}\label{secdisc}
In this section, we discuss some caveats of our approach that may reduce its applicability and propose possible extensions. First, gPC-DDP requires the system dynamics and cost function to be differentiable. If discontinuities appear in the problem formulation, one could use existing smoothing methodologies, as the ones presented in \cite{disc}.
Another  weakness of the gPC formulation is that its accuracy generally degrades over time. Moreover, estimating higher order moments often comes with greater error than estimating lower order moments \cite{fisher}. This happens because the underlying probability space is represented by a finite dimensional approximation. One can improve performance by adding more gPC coefficients in the expansions; however, this will only postpone the problem. This issue can be tackled by using multi-element methods \cite{mgpc}. This methodology can additionally deal with discontinous random inputs.

Finally, another limitation of gPC theory is that computational complexity increases with the dimensionality of the random inputs. In particular, the number of coefficients in  eq. \eqref{gpcexpansion1} increases exponentially with the dimension of the stochastic variables. Complexity increases also for full-tensor Gaussian quadrature formulas, as the one in \eqref{gaussquad}, since all possible combinations of the random parameters have to be considered in the summation. In fact, this was noticed to be the major drawback in our algorithm, since gPC-DDP uses Gaussian quadrature at each time instant both in the linearization and propagation phase. However, as explained in section \ref{gggb}, sparse quadrature formulas can be utilized to curtail the number of quadrature nodes \cite{smol}. Moreover, a sparse gPC approximation can be used for multi-dimensional gPC expansions (e.g., as the one presented in \cite{spam}) to reduce the number of coefficients. Both of these methodologies can significantly reduce computational complexity, while retaining sufficient accuracy.

\section{Conclusion}\label{secconc}
 In this paper we proposed a new methodology for stochastic trajectory optimization, for systems with uncertainty in the parameters and initial states. Polynomial Chaos played a key role since it allowed us to  handle stochasticity for a broad class of distributions. The developed optimal control scheme was based on Dynamic Programming principles and its main feature is the ability to control the statistical behavior of stochastic, high-dimensional, nonlinear systems. It was also proven that when some mild assumptions are satisfied, our iterative algorithm converges to a solution with locally quadratic convergence rates. Lastly, we showed that Variational Integrators can be formulated when Polynomial Chaos representations of mechanical systems are considered. Simulated examples verified the potential applicability of the proposed framework and highlighted the benefits of variational integration schemes in stochastic control tasks.

\appendices

\section{Computation of the second-order terms in the DEL linearization scheme}\label{appdel}
Here, we define the entries of the matrices $\Gamma^k_{DEL}$, $\Delta^k_{DEL}$, $\Xi^k_{DEL}$ and $\Lambda^k_{DEL}$ of eq. \eqref{dellin}. Specifically, the second-order derivatives of $\bold{Q}^{k+1}$ and $\hat{\bold{P}}^{k+1}$ will be determined with respect to $\bold{Q}^{k}$, $\hat{\bold{P}}^{k}$, $u^k$. The derivation is based on \cite{lin} and the results are included here for completeness. In the expressions below, we use $\times_i$ to denote the ``mode-$i$'' multiplication between a tensor and a matrix.

\scalebox{0.8}{\parbox{\linewidth}{
\begin{flalign*}
&\frac{\partial^2 \bold{Q}^{k+1}}{\partial \bold{Q}^k\partial\bold{Q}^k}=\bigg(\big[(D_2D_2D_1\hat L^k+D_2D_2\hat{\textbf{F}}^{k-})\times_3\bigg(\frac{\partial \bold{Q}^{k+1}}{\partial \bold{Q}^k}\bigg)^\top\big]\times_2\bigg(\frac{\partial \bold{Q}^{k+1}}{\partial \bold{Q}^k}\bigg)^\top\\
&+(D_2D_1D_1\hat L^k+D_2D_1\hat{\textbf{F}}^{k-})\times_3\bigg(\frac{\partial \bold{Q}^{k+1}}{\partial \bold{Q}^k}\bigg)^\top +D_1D_1D_1\hat L^k+D_1D_1\hat{\textbf{F}}^{k-}\\
&+(D_1D_2D_1\hat L^k+D_1D_2\hat{\textbf{F}}^{k-})\times_2\bigg(\frac{\partial \bold{Q}^{k+1}}{\partial \bold{Q}^k}\bigg)^\top
\bigg)\times_1(-\hat M^k)^{-1},&
\end{flalign*}}}

\scalebox{0.8}{\parbox{\linewidth}{
\begin{flalign*}
&\frac{\partial^2 \bold{Q}^{k+1}}{\partial \bold{Q}^k\partial\hat{\bold{P}}^k}=\bigg(\big[(D_2D_2D_1\hat L^k+D_2D_2\hat{\textbf{F}}^{k-})\times_3\bigg(\frac{\partial \bold{Q}^{k+1}}{\partial \bold{Q}^k}\bigg)^\top\big]\times_2\bigg(\frac{\partial \bold{Q}^{k+1}}{\partial \hat{\bold{P}}^k}\bigg)^\top\\
&+(D_2D_1D_1\hat L^k+D_2D_1\hat{\textbf{F}}^{k-})\times_3\bigg(\frac{\partial \bold{Q}^{k+1}}{\partial \hat{\bold{P}}^k}\bigg)^\top\bigg)\times_1(-\hat M^k)^{-1},&
\end{flalign*}}}

\scalebox{0.8}{\parbox{\linewidth}{
\begin{flalign*}
&\frac{\partial^2 \bold{Q}^{k+1}}{\partial \hat{\bold{P}}^k\partial\hat{\bold{P}}^k}=\bigg(\big[(D_2D_2D_1\hat L^k+D_2D_2\hat{\textbf{F}}^{k-})\times_3\bigg(\frac{\partial \bold{Q}^{k+1}}{\partial \hat{\bold{P}}^k}\bigg)^\top\big]\times_2\\
&\bigg(\frac{\partial \bold{Q}^{k+1}}{\partial \hat{\bold{P}}^k}\bigg)^\top\bigg)\times_1(-\hat M^k)^{-1},&
\end{flalign*}}}

\scalebox{0.8}{\parbox{\linewidth}{
\begin{flalign*}
&\frac{\partial^2 \bold{Q}^{k+1}}{\partial u^k\partial \bold{Q}^k}=\bigg(\big[(D_2D_2D_1\hat L^k+D_2D_2\hat{\textbf{F}}^{k-})\times_3\bigg(\frac{\partial \bold{Q}^{k+1}}{\partial u^k}\bigg)^\top\big]\times_2\bigg(\frac{\partial \bold{Q}^{k+1}}{\partial \bold{Q}^k}\bigg)^\top\\
&+D_3D_1\hat{\textbf{F}}^{k-}+D_3D_2\hat{\textbf{F}}^{k-}\times_2\bigg(\frac{\partial \bold{Q}^{k+1}}{\partial \bold{Q}^k}\bigg)^\top+&\\
&(D_2D_1D_1\hat L^k+D_2D_1\hat{\textbf{F}}^{k-})\times_3\bigg(\frac{\partial \bold{Q}^{k+1}}{\partial u^k}\bigg)^\top\bigg)\times_1(-\hat M^k)^{-1},&
\end{flalign*}}}

\scalebox{0.8}{\parbox{\linewidth}{
\begin{flalign*}
&\frac{\partial^2 \bold{Q}^{k+1}}{\partial u^k\partial \hat{\bold{P}}^k}=\bigg(\big[(D_2D_2D_1\hat L^k+D_2D_2\hat{\textbf{F}}^{k-})\times_3\bigg(\frac{\partial \bold{Q}^{k+1}}{\partial u^k}\bigg)^\top\big]\times_2\bigg(\frac{\partial \bold{Q}^{k+1}}{\partial \hat{\bold{P}}^k}\bigg)^\top\\
&+D_3D_2\hat{\textbf{F}}^{k-}\times_2\bigg(\frac{\partial \bold{Q}^{k+1}}{\partial \hat{\bold{P}}^k}\bigg)^\top\bigg)\times_1(-\hat M^k)^{-1},&
\end{flalign*}
}}

\scalebox{0.8}{\parbox{\linewidth}{
\begin{flalign*}
&\frac{\partial^2 \bold{Q}^{k+1}}{\partial u^k\partial u^k}=\bigg(\big[(D_2D_2D_1\hat L^k+D_2D_2\hat{\textbf{F}}^{k-})\times_3\bigg(\frac{\partial \bold{Q}^{k+1}}{\partial u^k}\bigg)^\top\big]\times_2\bigg(\frac{\partial \bold{Q}^{k+1}}{\partial u^k}\bigg)^\top\\
&+D_3D_2\hat{\textbf{F}}^{k-}\times_2\bigg(\frac{\partial \bold{Q}^{k+1}}{\partial u^k}\bigg)^\top+D_2D_3\hat{\textbf{F}}^{k-}\times_3\bigg(\frac{\partial \bold{Q}^{k+1}}{\partial u^k}\bigg)^\top+\\
&D_3D_3\hat{\textbf{F}}^{k-}\bigg)\times_1(-\hat M^k)^{-1},&
\end{flalign*}}}

\scalebox{0.8}{\parbox{\linewidth}{
\begin{flalign*}
&\frac{\partial^2 \hat{\bold{P}}^{k+1}}{\partial \textbf Q^k\partial \textbf Q^k}=\big[(D_2D_2D_2\hat L^k+D_2D_2\hat{\textbf{F}}^{k+})\times_3\bigg(\frac{\partial \bold{Q}^{k+1}}{\partial \bold{Q}^k}\bigg)^\top+D_1D_2D_2\hat L^k+\\
&D_1D_2\hat{\textbf{F}}^{k+}\big]\times_2\bigg(\frac{\partial \bold{Q}^{k+1}}{\partial \bold{Q}^k}\bigg)^\top+\frac{\partial^2 \bold{Q}^{k+1}}{\partial \textbf Q^k\partial \textbf Q^k}\times_1(D_2D_2\hat L^k+D_2\hat{\textbf{F}}^{k+})+\\
&(D_2D_1D_2\hat L^k+D_2D_1\hat{\textbf{F}}^{k+})\times_3\bigg(\frac{\partial \bold{Q}^{k+1}}{\partial \textbf Q^k}\bigg)^\top+D_1D_1D_2\hat L^k+D_1D_1\hat{\textbf{F}}^{k+},&
\end{flalign*}}}

\scalebox{0.8}{\parbox{\linewidth}{
\begin{flalign*}
&\frac{\partial^2 \hat{\bold{P}}^{k+1}}{\partial \textbf Q^k\partial \hat{\textbf P}^k}=\frac{\partial^2 \bold{Q}^{k+1}}{\partial \textbf Q^k\partial \hat{\textbf P}^k}\times_1(D_2D_2\hat L^k+D_2\hat{\textbf{F}}^{k+})+&\\
&(D_1D_2D_2\hat L^k+D_1D_2\hat{\textbf{F}}^{k+})\times_2\bigg(\frac{\partial \bold{Q}^{k+1}}{\partial \hat{\textbf P}^k}\bigg)^\top+&\\
&\big[(D_2D_2D_2\hat L^k+D_2D_2\hat{\textbf{F}}^{k+})\times_3\bigg(\frac{\partial \bold{Q}^{k+1}}{\partial \bold{Q}^k}\bigg)^\top\big]\times_2\bigg(\frac{\partial \bold{Q}^{k+1}}{\partial \hat{\bold{P}}^k}\bigg)^\top,&
\end{flalign*}}}

\scalebox{0.8}{\parbox{\linewidth}{
\begin{flalign*}
&\frac{\partial^2 \hat{\bold{P}}^{k+1}}{\partial \hat{\textbf P}^k\partial \hat{\textbf P}^k}=\big[(D_2D_2D_2\hat L^k+D_2D_2\hat{\textbf{F}}^{k+})\times_3\bigg(\frac{\partial \bold{Q}^{k+1}}{\partial \hat{\bold{P}}^k}\bigg)^\top\big]\times_2\bigg(\frac{\partial \bold{Q}^{k+1}}{\partial \hat{\bold{P}}^k}\bigg)^\top+\\
&\frac{\partial^2 \bold{Q}^{k+1}}{\partial \hat{\textbf P}^k\partial \hat{\textbf P}^k}\times_1(D_2D_2\hat L^k+D_2\hat{\textbf{F}}^{k+}),&
\end{flalign*}}}

\scalebox{0.8}{\parbox{\linewidth}{
\begin{flalign*}
&\frac{\partial^2 \hat{\bold{P}}^{k+1}}{\partial u^k\partial \textbf Q^k}=\big[(D_2D_2D_2\hat L^k+D_2D_2\hat{\textbf{F}}^{k+})\times_3\bigg(\frac{\partial \bold{Q}^{k+1}}{\partial u^k}\bigg)^\top\big]\times_2\bigg(\frac{\partial \bold{Q}^{k+1}}{\partial \bold{Q}^k}\bigg)^\top+\\
&D_3D_2\hat{\textbf{F}}^{k+}\times_2\bigg(\frac{\partial \bold{Q}^{k+1}}{\partial \textbf Q^k}\bigg)^\top+\frac{\partial^2 \bold{Q}^{k+1}}{\partial u^k\partial \textbf Q^k}\times_1(D_2D_2\hat L^k+D_2\hat{\textbf{F}}^{k+})+&\\
&(D_2D_1D_2\hat L^k+D_2D_1\hat{\textbf{F}}^{k+})\times_3\bigg(\frac{\partial \bold{Q}^{k+1}}{\partial u^k}\bigg)^\top+D_3D_1\hat{\textbf{F}}^{k+},&
\end{flalign*}}}

\scalebox{0.8}{\parbox{\linewidth}{
\begin{flalign*}
&\frac{\partial^2 \hat{\bold{P}}^{k+1}}{ \partial u^k\partial\hat{\textbf P}^k}=D_3D_2\hat{\textbf{F}}^{k+}\times_2\bigg(\frac{\partial \bold{Q}^{k+1}}{\partial \hat{\textbf P}^k}\bigg)^\top+\frac{\partial^2 \bold{Q}^{k+1}}{\partial u^k\partial \hat{\textbf P}^k}\times_1(D_2D_2\hat L^k+D_2\hat{\textbf{F}}^{k+})+&\\
&\big[(D_2D_2D_2\hat L^k+D_2D_2\hat{\textbf{F}}^{k+})\times_3\bigg(\frac{\partial \bold{Q}^{k+1}}{\partial u^k}\bigg)^\top\big]\times_2\bigg(\frac{\partial \bold{Q}^{k+1}}{\partial \hat{\bold{P}}^k}\bigg)^\top,&
\end{flalign*}}}

\scalebox{0.8}{\parbox{\linewidth}{
\begin{flalign*}
&\frac{\partial^2 \hat{\bold{P}}^{k+1}}{\partial u^k\partial u^k}=D_3D_3\hat{\textbf{F}}^{k+}+\frac{\partial^2 \bold{Q}^{k+1}}{\partial u^k\partial u^k}\times_1(D_2D_2\hat L^k+D_2\hat{\textbf{F}}^{k+})+&\\
&D_3D_2\hat{\textbf{F}}^{k+}\times_2\bigg(\frac{\partial \bold{Q}^{k+1}}{\partial u^k}\bigg)^\top+D_2D_3\hat{\textbf{F}}^{k+}\times_3\bigg(\frac{\partial \bold{Q}^{k+1}}{\partial u^k}\bigg)^\top+&\\
&\big[(D_2D_2D_2\hat L^k+D_2D_2\hat{\textbf{F}}^{k+})\times_3\bigg(\frac{\partial \bold{Q}^{k+1}}{\partial u^k}\bigg)^\top\big]\times_2\bigg(\frac{\partial \bold{Q}^{k+1}}{\partial u^k}\bigg)^\top.&
\end{flalign*}}}

\section{Convergence rate of gPC-DDP}\label{appconv}
Here, we provide the proof for Theorem \ref{thconv}. Denote by $U^{(l)}$ the controls given by gPC-DDP at its $l^{\text{th}}$ iteration, i.e., $U^{(l)}=\text{vec}[u^{0(l)},u^{1(l)},...,u^{K_f-1(l)}]\in\mathbb{R}^{m(K_f-1)}$, with $K_f$ being the time horizon and $\text{vec}[\cdot]$ the ``vec'' operator. Moreover, we denote the nominal control sequence by $\bar{U}=U^{(l-1)}$ and the updated one by $U=U^l$. Let also $U^*$ be the (sub)optimal solution that gPC-DDP converges to and $\textbf{X}^*$ the corresponding state trajectory.

The authors in \cite{liao000} showed that when $l$ becomes large enough, the line search parameter of the deterministic DDP can be set equal to 1. The same argument can be used for $\gamma$ in the controls update \eqref{newu} of gPC-DDP (proof is omitted). Hence, we will pick $\gamma=1$ for our analysis.

The goal is first to show that \eqref{convrate} is satisfied when $K_f=2$ and then generalize for arbitrary $K_f$. For $K_f=2$ the optimal control problem \eqref{gpcstochoptprob3} becomes in discrete time

\begin{equation*}
	\begin{split}
	\min_{u_0,u_1}\hspace{1.7mm}\bm J=\min_{u_0,u_1}\hspace{1.7mm}\bm L^0(\textbf X^0,u^0)+\bm L^1(\textbf X^1,u^1&)+\bm F(\textbf{X}^2)\\
	\text{s.t.}\quad\bold{X}^1=\bold f^0(\bold{X}^0,u^0),\quad\bold{X}^2=\bold f^1(\bold{X}^1,u^1)&,\quad\bold{X}^0 = \overline{\bold{X}}^0,
	\end{split}
	\end{equation*}
where $\textbf{f}^k$ is a discretized version of \eqref{gpcX}. Henceforth, we will write for convenience $\cdot|_{\textunderscore}$ or $\cdot|_*$ when the quantity ``$\cdot$" is evaluated at the nominal or (sub)optimal trajectory respectively. Implementing gPC-DDP as presented in section \ref{secgpcddp}, yields

\begin{itemize}
\item for $k=1$:
\end{itemize}
\begin{equation}
\label{qxx1_proof}
Q_{\textbf{xx}}^1=\bm{L}^1_{\textbf{xx}}|_{\textunderscore}+(\Theta^1)^\top\bm F_{\textbf{xx}}|_{\textunderscore}\Theta^1
+\sum_{i=1}^{\bm n}\big([\bm F_{\textbf{x}}]_i\nabla_{\textbf{x}\textbf{x}}[\textbf{f}^1]_i\big)|_{\textunderscore},
\end{equation}
\begin{equation}
\label{quu1_proof}
Q_{uu}^1=\bm{L}^1_{uu}|_{\textunderscore}+(B^1)^\top\bm F_{\textbf{xx}}|_{\textunderscore}B^1
+\sum_{i=1}^{\bm n}\big([\bm F_{\textbf{x}}]_i\nabla_{uu}[\textbf{f}^1]_i\big)|_{\textunderscore},
\end{equation}
\begin{equation}
\label{qxu1_proof}
Q_{\textbf{x}u}^1=(Q_{u\textbf{x}}^1)^\top=\bm{L}^1_{\textbf{x}u}|_{\textunderscore}+(\Theta^1)^\top\bm F_{\textbf{xx}}|_{\textunderscore}B^1
+\sum_{i=1}^{\bm n}\big([\bm F_{\textbf{x}}]_i\nabla_{\textbf{x}u}[\textbf{f}^1]_i\big)|_{\textunderscore},
\end{equation}
\begin{equation}
\label{qu1_proof}
Q_{u}^1=\bm{L}^1_{u}|_{\textunderscore}+(B^1)^\top\bm F_{\textbf{x}}|_{\textunderscore},
\end{equation}
\begin{equation}
\label{qx1_proof}
Q_{\textbf x}^1=\bm{L}^1_{\textbf x}|_{\textunderscore}+(\Theta^1)^\top\bm F_{\textbf{x}}|_{\textunderscore},
\end{equation}
\begin{equation}
\label{l1_proof}
\ell^1=(-Q_{uu}^1)^{-1}Q_u^1,
\end{equation}
\begin{equation}
\label{s1_proof}
\Sigma^1=(-Q_{uu}^1)^{-1}Q_{u\textbf x}^1,
\end{equation}
\begin{equation}
\label{vx1_proof}
V_{\textbf{x}}^1=Q_{\textbf{x}}^1+Q_{\textbf{x}u}^1\ell^1,
\end{equation}
\begin{equation}
\label{vxx1_proof}
V_{\textbf{xx}}^1=Q_{\textbf{xx}}^1+Q_{\textbf{x}u}^1\Sigma^1,
\end{equation}
\begin{equation}
\label{u1_proof}
u^1=\bar{u}^1+\ell^1+\Sigma^1\delta\textbf{X}^1,
\end{equation}

\begin{itemize}
\item for $k=0$:
\end{itemize}
\begin{equation}
\begin{split}
\label{quu0_proof}
Q_{uu}^0=&\bm{L}^0_{uu}|_{\textunderscore}+(B^0)^\top\big(Q_{\textbf{xx}}^1-Q_{\textbf{x}u}^1(Q_{uu}^1)^{-1}Q_{u\textbf x}^1\big)B^0\\
&+\sum_{i=1}^{\bm n}\big([Q_{\textbf{x}}^1-Q_{\textbf{x}u}^1(Q_{uu}^1)^{-1}Q_{u}^1]_i\nabla_{uu}[\textbf{f}^0]_i\big)\big|_{\textunderscore},
\end{split}
\end{equation}
\begin{equation}
\label{qu0_proof}
Q_{u}^0=\bm{L}^0_{u}|_{\textunderscore}+(B^0)^\top\big(Q_{\textbf{x}}^1-Q_{\textbf{x}u}^1(Q_{uu}^1)^{-1}Q_{u}^1\big),
\end{equation}
\begin{equation}
\label{l0_proof}
\ell^0=(-Q_{uu}^0)^{-1}Q_u^0,
\end{equation}
\begin{equation}
\label{u0_proof}
u^0=\bar{u}^0+\ell^0.
\end{equation}
Moreover, \eqref{deltax} gives
\begin{equation}
\label{deltax1}
\delta\textbf X^1=B^0\delta u^0+O(||\delta u^0||^2),
\end{equation}
\begin{equation}
\label{deltax2}
\delta\textbf X^2=\Theta^1B^0\delta u^0+B^1\delta u^1+O(||\delta U||^2).
\end{equation}
In the expressions above we have used that $\delta\bold{X}^0=0$. Now, since $U^*$ is a stationary point we have
\begin{equation*}
\nabla_{u^1}\bm J|_*=0
\end{equation*}
\begin{equation}
\label{gradJ_u1}
\Rightarrow\big(\bm{L}^1_{u}+(\nabla_{u}\textbf{f}^1)^\top\bm F_{\textbf{x}}\big)\big|_*=0,
\end{equation}
\begin{equation*}
\nabla_{u^0}\bm J|_*=0
\end{equation*}
\begin{equation}
\label{gradJ_u0}
\Rightarrow\big(\bm{L}^0_{u}+(\nabla_{u}\textbf{f}^0)^\top\bm{L}^1_{\textbf x}+(\nabla_{u}\textbf{f}^0)^\top(\nabla_{\textbf x}\textbf{f}^1)^\top\bm F_{\textbf{x}}\big)\big|_*=0.
\end{equation}
Next, we expand $\nabla_{u^1}\bm J|_*$ about the nominal trajectory
\begin{equation*}
\begin{split}
&\big(\bm{L}^1_{u}+(\nabla_{u}\textbf{f}^1)^\top\bm F_{\textbf{x}}\big)\big|_*-\big(\bm{L}^1_{u}+(\nabla_{u}\textbf{f}^1)^\top\bm F_{\textbf{x}}\big)\big|_{\textunderscore}=\\
&\bm{L}^1_{uu}|_{\textunderscore}(u^{*1}-\bar{u}^1)+\bigg(\bm{L}^1_{u\textbf x}+\sum_{i=1}^{\bm n}\big([\bm F_{\textbf{x}}]_i\nabla_{u\textbf{x}}[\textbf{f}^1]_i\big)\bigg)\bigg|_{\textunderscore}(\textbf X^{*1}-\overline{\textbf X}^1)+\\
&\sum_{i=1}^{\bm n}\big([\bm F_{\textbf{x}}]_i\nabla_{uu}[\textbf{f}^1]_i\big)|_{\textunderscore}(u^{*1}-\bar{u}^1)+\big((\nabla_u\textbf f^1)^\top\bm F_{\textbf{xx}}\big)\big|_{\textunderscore}(\textbf X^{*2}-\overline{\textbf X}^2)+\\
&O\bigg(\left\Vert\big((\textbf X^{*1})^\top,(\textbf X^{*2})^\top,(u^{*1})^\top\big)-\big((\overline{\textbf X}^1)^\top,(\overline{\textbf X}^2)^\top,(\bar{u}^1)^\top\big)\right\Vert^2\bigg).
\end{split}
\end{equation*}
Rearranging and using \eqref{quu1_proof} -- \eqref{qu1_proof} and \eqref{deltax1} -- \eqref{gradJ_u1} yields
\begin{equation*}
\begin{split}
&\underbrace{\big(\bm{L}^1_{u}+(B^1)^\top\bm F_{\textbf{x}}\big)\big|_{\textunderscore}}_{Q_u^1}=\\
&\underbrace{\big(\bm{L}^1_{uu}+\sum_{i=1}^{\bm n}[\bm F_{\textbf{x}}]_i\nabla_{uu}[\textbf{f}^1]_i+(B^1)^\top\bm F_{\textbf{xx}}B^1\big)\big|_{\textunderscore}}_{Q_{uu}^1}(\bar{u}^1-u^{*1})+\\
&\underbrace{\big(\bm{L}^1_{u\textbf x}+\sum_{i=1}^{\bm n}[\bm F_{\textbf{x}}]_i\nabla_{u\textbf x}[\textbf{f}^1]_i+(B^1)^\top\bm F_{\textbf{xx}}\Theta^1\big)\big|_{\textunderscore}}_{Q_{u\textbf x}^1}B^0(\bar{u}^0-u^{*0})\\
&+O(||\bar{U}-U^*||^2)
\end{split}
\end{equation*}
\begin{equation}
\label{expr_19}
\Rightarrow Q_u^1=Q_{uu}^1(\bar{u}^1-u^{*1})+Q_{u\textbf x}^1B^0(\bar{u}^0-u^{*0})+O(||\bar{U}-U^*||^2).
\end{equation}
Similarly, $\nabla_{u^0}\bm J|_*$ is expanded as follows
\begin{equation*}
\begin{split}
&\big(\bm{L}^0_{u}+(\nabla_{u}\textbf{f}^0)^\top\bm{L}^1_{\textbf x}+(\nabla_{u}\textbf{f}^0)^\top(\nabla_{\textbf x}\textbf{f}^1)^\top\bm F_{\textbf{x}}\big)\big|_*-\\
&\big(\bm{L}^0_{u}+(\nabla_{u}\textbf{f}^0)^\top\bm{L}^1_{\textbf x}+(\nabla_{u}\textbf{f}^0)^\top(\nabla_{\textbf x}\textbf{f}^1)^\top\bm F_{\textbf{x}}\big)\big|_{\textunderscore}=\bm{L}^0_{uu}|_{\textunderscore}(u^{*0}-\bar{u}^0)\\
&+\sum_{i=1}^{\bm n}([\bm{L}^1_{\textbf x}]_i\nabla_{uu}[\textbf{f}^0]_i)|_{\textunderscore}(u^{*0}-\bar{u}^0)+\big((\nabla_u\textbf f^0)^\top)\bm{L}^1_{\textbf {xx}}\big)\big|_{\textunderscore}(\textbf X^{*1}-\overline{\textbf X}^1)+\\
&\big((\nabla_u\textbf f^0)^\top\bm{L}^1_{\textbf xu}\big)\big|_{\textunderscore}(u^{*1}-\bar{u}^1)+\\
&\sum_{i=1}^{\bm n}\big([(\nabla_{\textbf x}\textbf{f}^1)^\top\bm F_{\textbf{x}}]_i \nabla_{uu}[\textbf{f}^0]_i\big)\big|_{\textunderscore}(u^{*0}-\bar{u}^0)+\\
&\big((\nabla_{u}\textbf{f}^0)^\top\sum_{i=1}^{\bm n}[\bm F_{\textbf{x}}]_i\nabla_{\textbf {xx}}[\textbf{f}^{1}]_i\big)\big|_{\textunderscore}(\textbf X^{*1}-\overline{\textbf X}^1)+\\
&\big((\nabla_{u}\textbf{f}^0)^\top\sum_{i=1}^{\bm n}[\bm F_{\textbf{x}}]_i\nabla_{\textbf {x}u}[\textbf{f}^{1}]_i\big)\big|_{\textunderscore}(u^{*1}-\bar{u}^1)+\\
&\big((\nabla_{u}\textbf{f}^0)^\top(\nabla_{\textbf x}\textbf{f}^1)^\top\bm F_{\textbf{xx}}\big)\big|_{\textunderscore}(\textbf X^{*2}-\overline{\textbf X}^2)+\\
&O\bigg(\left\Vert\big((\textbf X^{*1})^\top,(\textbf X^{*2})^\top,(U^{*})^\top\big)-\big((\overline{\textbf X}^1)^\top,(\overline{\textbf X}^2)^\top,(\bar{U})^\top\big)\right\Vert^2\bigg).
\end{split}
\end{equation*}
Rearranging and using \eqref{qxx1_proof}, \eqref{qxu1_proof}, \eqref{qx1_proof}, \eqref{deltax1}, \eqref{deltax2}, \eqref{gradJ_u0} gives
\begin{equation*}
\begin{split}
&\big(\bm{L}^0_{u}+(B^0)^\top\underbrace{(\bm{L}^1_{\textbf x}+(\Theta^1)^\top\bm F_{\textbf{x}})}_{Q_{\textbf x}^1}\big)\big|_{\textunderscore}=\\
&\bigg(\bm{L}^0_{uu}+\sum_{i=1}^{\bm n}\big[\underbrace{\bm{L}^1_{\textbf x}+(\Theta^1)^\top\bm F_{\textbf{x}}}_{Q_{\textbf x}^1}\big]_i\nabla_{uu}[\textbf{f}^0]_i+\\
&(B^0)^\top\big(\underbrace{\bm{L}^1_{\textbf {xx}}+\sum_{i=1}^{\bm n}[\bm F_{\textbf{x}}]_i\nabla_{\textbf {xx}}[\textbf{f}^{1}]_i+(\Theta^1)^\top\bm F_{\textbf{xx}}\Theta^1\big)}_{Q_{\textbf{xx}}^1}B^0\bigg)\bigg|_{\textunderscore}(\bar{u}^0-u^{*0})\\
&+(B^0)^\top\underbrace{\bigg(\bm{L}^1_{\textbf xu}+\sum_{i=1}^{\bm n}[\bm F_{\textbf{x}}]_i\nabla_{\textbf {x}u}[\textbf{f}^{1}]_i+(\Theta^1)^\top\bm F_{\textbf{xx}}B^1\bigg)}_{Q_{\textbf xu}^1}\bigg|_{\textunderscore}(\bar{u}^1-u^{*1})\\
&+O(||\bar{U}-U^*||^2)
\end{split}
\end{equation*}
\begin{equation}
\label{expr_20}
\begin{split}
&\Rightarrow\big(\bm{L}^0_{u}+(B^0)^\top Q_{\textbf x}^1\big)\big|_{\textunderscore}=\\
&\bigg(\bm{L}^0_{uu}+\sum_{i=1}^{\bm n}\big[Q_{\textbf x}^1\big]_i\nabla_{uu}[\textbf{f}^0]_i+(B^0)^\top Q_{\textbf{xx}}^1B^0\bigg)\bigg|_{\textunderscore}(\bar{u}^0-u^{*0})\\
&+(B^0)^\top Q_{\textbf xu}^1(\bar{u}^1-u^{*1})+O(||\bar{U}-U^*||^2).
\end{split}
\end{equation}
Now, utilizing \eqref{l1_proof}, \eqref{s1_proof}, \eqref{u1_proof}, \eqref{deltax1} and \eqref{expr_19}, the new control at $k=1$ is obtained as follows
\begin{equation*}
\begin{split}
u^1=&\bar{u}^1+(u^{*1}-\bar{u}^1)+(Q_{uu}^1)^{-1}Q_{u\textbf x}^1B^0(u^{*0}-\bar{u}^0)+\\
&(Q_{uu}^1)^{-1}Q_{u\textbf{x}}^1B^0(\bar{u}^0-u^{0})+O(||U^*-\bar{U}||^2)
\end{split}
\end{equation*}
\begin{equation}
\label{expr_23}
\Rightarrow u^1-u^{*1}=-(Q_{uu}^1)^{-1}Q_{u\textbf x}^1B^0(\bar{u}^0-u^{*0})+O(||\bar{U}-U^*||^2).
\end{equation}
Additionally, we plug \eqref{expr_19} into \eqref{quu0_proof} and get the following expression
\begin{equation}
\begin{split}
\label{expr_24}
&Q_{uu}^0=\bm{L}^0_{uu}|_{\textunderscore}+(B^0)^\top\big(Q_{\textbf{xx}}^1-Q_{\textbf{x}u}^1(Q_{uu}^1)^{-1}Q_{u\textbf x}^1\big)B^0+\\
&\sum_{i=1}^{\bm n}\bigg([Q_{\textbf{x}}^1-Q_{\textbf{x}u}^1(\bar{u}^1-u^{*1})-\\
&Q_{\textbf{x}u}^1(Q_{uu}^1)^{-1}Q_{u\textbf x}^1B^0(\bar{u}^0-u^{*0})]_i\nabla_{uu}[\textbf{f}^0]_i\bigg)\bigg|_{\textunderscore}\\
&+O(||\bar{U}-U^*||^2).
\end{split}
\end{equation}
Substituting \eqref{expr_23} into \eqref{expr_24} gives
\begin{equation}
\begin{split}
\label{expr_24_new}
&Q_{uu}^0=\bm{L}^0_{uu}|_{\textunderscore}+(B^0)^\top\big(Q_{\textbf{xx}}^1-Q_{\textbf{x}u}^1(Q_{uu}^1)^{-1}Q_{u\textbf x}^1\big)B^0+\\
&\sum_{i=1}^{\bm n}\big([Q_{\textbf{x}}^1]_i\nabla_{uu}[\textbf{f}^0]_i\big)\big|_{\textunderscore}+O(||\bar{U}-U^*||^2).
\end{split}
\end{equation}
In a similar manner, we will get a new expression for $Q_u^0$. Specifically, plugging \eqref{expr_19} into \eqref{qu0_proof} yields
\begin{equation}
\label{expr_25}
\begin{split}
Q_{u}^0=&\bm{L}^0_{u}|_{\textunderscore}+(B^0)^{\top}Q_{\textbf{x}}^1-(B^0)^{\top}Q_{\textbf{x}u}^1|_{\textunderscore}(\bar{u}^1-u^{1*})-\\
&(B^0)^{\top}Q_{\textbf{x}u}^1(Q_{uu}^1)^{-1}Q_{u\textbf x}^1B^0(\bar{u}^0-u^{*0})+O(||\bar{U}-U^*||^2).
\end{split}
\end{equation}
Substituting the right-hand side of \eqref{expr_20} into \eqref{expr_25} gives
\begin{equation}
	\begin{split}
	\label{expr_26}
	&Q_{u}^0=\big(\bm{L}^0_{uu}+(B^0)^{\top}(Q_{\textbf{xx}}^1-Q_{\textbf{x}u}^1(Q_{uu}^1)^{-1}Q_{u\textbf x}^1)B^0\\
	&\sum_{i=1}^n\big[Q_{\textbf x}^1\big]_i\nabla_{uu}[\textbf{f}^0]_i\big)\big|_{\textunderscore}(\bar{u}^0-u^{*0})+O(||\bar{U}-U^*||^2).
	\end{split}
	\end{equation}
	Finally, from \eqref{l0_proof}, \eqref{u0_proof}, \eqref{expr_24_new} and \eqref{expr_26} the control at $k=0$ will satisfy
\begin{equation}
\label{u0_quad}
u^0-u^{*0}=O(||\bar{U}-U^*||^2)
\end{equation}
and hence from \eqref{expr_23}
\begin{equation}
\label{u1_quad}
u^1-u^{*1}=O(||\bar{U}-U^*||^2).
\end{equation}
Expressions \eqref{u0_quad} and \eqref{u1_quad} imply that $\exists$ $c>0$ such that for $K_f=2$, the following is true
\begin{equation}
\label{expr_quad}
||U-U^*||\leq c||\bar{U}-U^*||^2,
\end{equation}
i.e., the convergence rate is greater than or equal to quadratic.

Now assume that \eqref{expr_quad} is true for the generic problem
\begin{equation}
\label{j_h}
 \min_U \bigg[\sum_{i=0}^{H-1}L^i(\textbf X^i,u^i)+\bm F(\textbf{X}^H)\bigg].
\end{equation}
The goal is to show that a similar expression to \eqref{expr_quad} can be obtained for $K_f=H+1$. In that case, the cost function is written as
\begin{equation*}
J= \sum_{i=0}^{H-1}L^i(\textbf X^i,u^i)+L^H(\textbf X^H,u^H)+\bm F(\textbf{X}^{H+1}).
\end{equation*}
By utilizing the dynamics and controls given by gPC-DDP at $k=H$, the above expression is equivalent to
\begin{equation}
\label{j_h1}
\begin{split}
&J= \bm F\big(\textbf{f}^{H}(\textbf X^{H},\bar{u}^H+\ell^H+\Sigma_H(\textbf X^{H}-\overline{\textbf X}^H)\big)+\sum_{i=0}^{H-1}L^i(\textbf X^i,u^i)+\\
&L^H\big(\textbf{f}^{H-1}(\textbf X^{H-1},u^{H-1}),\bar{u}^H+\ell^H+\\
&\Sigma_H(\textbf{f}^{H-1}(\textbf X^{H-1},u^{H-1})-\overline{\textbf X}^H\big).
\end{split}
\end{equation}
Let $U(H-1)$ represent the controls up to time instant $H-1$ (i.e., $U(H-1)=(u^0,...,u^{H-1})^\top$). Since gPC-DDP is based on Bellman's principle of optimality \eqref{v17}, one can readily show that solving problems \eqref{j_h} and \eqref{j_h1} for fixed nominal trajectories $\overline{X}$, produces the same control sequences $U(H-1)$. Hence, by assumption, $\exists$ $c_{H-1}>0$ for which
\begin{equation*}
||U(H-1)-U(H-1)^*||\leq c_{H-1}||\bar{U}(H-1)-U(H-1)^*||^2,
\end{equation*}
It remains to show that $\exists$ $c_H>0$ such that the control at $k=H$ satisfies
\begin{equation*}
||u^H-u^{*H}||\leq c_H||\bar{u}^H-u^{*H}||^2.
\end{equation*}
Towards that goal we can use an approach similar to Newton's method (e.g, \cite{liao000}, \cite{Nocedal}). Specifically
\begin{equation*}
\begin{split}
&||u^H-u^{*H}||=||\bar{u}^H+\ell^H+\Sigma_H(\textbf X^{H}-\overline{\textbf X}^H)-u^{*H}||\\
&\leq||\bar{u}^H-(Q_{uu}^H|_{\textunderscore})^{-1}Q_u^H|_{\textunderscore}-u^{*H}||\\
&=||(Q_{uu}^H|_{\textunderscore})^{-1}\big(Q_{uu}^H|_{\textunderscore}(\bar{u}^H-u^{*H})-(Q_u^H|_{\textunderscore}-\underbrace{Q_u^{H}|_*}_{=0})\big)||
\end{split}
\end{equation*}
Since by assumption the dynamics and cost functions are differentiable up to the third order, $Q_{u}^H$ will be second-order differentiable. Hence by using Taylor's formula one will have\footnote{In these particular expressions, $Q_u$ and $Q_{uu}$ are evaluated either at the nominal or the optimal trajectory, hence the notation $|_{\textunderscore}$ or $|_*$.}
\begin{equation*}
\begin{split}
||u^H-u^{*H}||&\leq||(Q_{uu}^H|_{\textunderscore})^{-1}||\hspace{1mm}||\big(Q_{uu}^H|_{\textunderscore}(\bar{u}^H-u^{*H})-\\
&\int_0^1Q_{uu}^H|_{\bar{u}^H+t(u^{*H}-\bar{u}^H)}(\bar{u}^H-u^{*H})\rd t||.
\end{split}
\end{equation*}
Using the Lipschitz continuity of $Q_u$ gives
\begin{equation*}
\begin{split}
&||u^H-u^{*H}||\leq\\
&||(Q_{uu}^H)^{-1}||\int_0^1||\underbrace{(Q_{uu}^H|_{\textunderscore}-Q_{uu}^H|_{\bar{u}^H+t(u^{*H}-\bar{u}^H)}}_{\leq\int_0^1Zt\rd t||\bar{u}^H-u^{*H}||}||\rd t||\bar{u}^H-u^{*H}||.
\end{split}
\end{equation*}
\begin{equation*}
\Rightarrow||u^H-u^{*H}||\leq c_H||\bar{u}^H-u^{*H}||^2,
\end{equation*}
where $Z$ is the Lipschitz constant and $c_H=\frac{1}{2}\max(||(Q_{uu}^H)^{-1}||)Z>0$. Note that $c_H$ is bounded since by assumption, $Q_{uu}^H$ is positive definite. Consequently, the proof is complete by induction.

\bibliographystyle{IEEEtran}
\bibliography{IEEEabrv,mybib}

\end{document}